\documentclass[11pt]{amsart}   	
\usepackage{geometry}             
\usepackage[dvipsnames]{xcolor}   		             		
\usepackage{graphicx}				
\usepackage{amssymb,mathrsfs}
\usepackage{amsthm,amsmath,stmaryrd}
\usepackage{tikz}
\usepackage{tikz-cd}
\usepackage{accents,enumerate}
\usepackage[headings]{fullpage}
\usepackage{bm}
\usepackage{mathtools}
\usepackage[all]{xy}
\usepackage{caption}
\usepackage[euler-digits]{eulervm}
\usepackage[normalem]{ulem}

\usepackage{verbatim}
\usepackage{booktabs}
\usepackage{float}

\tikzset{
	commutative diagrams/.cd, 
	arrow style=tikz, 
	diagrams={>=stealth}
}

\newtheorem{innercustomthm}{Theorem}
\newenvironment{customthm}[1]
{\renewcommand\theinnercustomthm{#1}\innercustomthm}
{\endinnercustomthm}

\newtheorem{innercustomcor}{Corollary}
\newenvironment{customcor}[1]
{\renewcommand\theinnercustomcor{#1}\innercustomcor}
{\endinnercustomcor}

\makeatletter
\def\@tocline#1#2#3#4#5#6#7{\relax
	\ifnum #1>\c@tocdepth 
	\else
	\par \addpenalty\@secpenalty\addvspace{#2}%
	\begingroup \hyphenpenalty\@M
	\@ifempty{#4}{%
		\@tempdima\csname r@tocindent\number#1\endcsname\relax
	}{%
		\@tempdima#4\relax
	}%
	\parindent\z@ \leftskip#3\relax \advance\leftskip\@tempdima\relax
	\rightskip\@pnumwidth plus4em \parfillskip-\@pnumwidth
	#5\leavevmode\hskip-\@tempdima
	\ifcase #1
	\or\or \hskip 1em \or \hskip 2em \else \hskip 3em \fi%
	#6\nobreak\relax
	\dotfill\hbox to\@pnumwidth{\@tocpagenum{#7}}\par
	\nobreak
	\endgroup
	\fi}
\makeatother

\usetikzlibrary{calc}
\usetikzlibrary{fadings}
\usetikzlibrary{decorations.pathmorphing}
\usetikzlibrary{decorations.pathreplacing}

\newcounter{marginnote}
\DeclareMathAlphabet{\mathpzc}{OT1}{pzc}{m}{it}

\usepackage[backref=page]{hyperref}
\hypersetup{
	colorlinks   = true,          
	urlcolor     = violet,          
	linkcolor    = blue,          
	citecolor   = violet             
}

\newtheorem{theorem}{Theorem}[subsection]
\newtheorem{corollary}[theorem]{Corollary}
\newtheorem{lemma}[theorem]{Lemma}
\newtheorem{proposition}[theorem]{Proposition}

\newtheorem{quasi-theorem}[theorem]{Quasi-Theorem}

\theoremstyle{definition}
\newtheorem{definition}[theorem]{Definition}
\newtheorem{warning}[theorem]{$\lightning$ Warning}
\newtheorem{remark}[theorem]{Remark}

\newtheorem{example}[theorem]{Example}
\newtheorem{blank remark}[theorem]{}

\newtheorem{not1}[theorem]{Notation}

\newcommand{\RR} {{\mathbb R}}		
\newcommand{\ZZ} {{\mathbb Z}}		

\def\setminus{\smallsetminus}

\newcommand{\Hom}{\operatorname{Hom}}

\DeclareMathOperator{\Aut}{Aut}

\DeclareMathOperator{\spec}{Spec}

\newcommand{\cal}{\mathcal}

\def\cA{{\cal A}}

\def\cD{{\cal D}}
\def\cE{{\cal E}}
\def\cF{{\cal F}}

\def\cM{{\cal M}}

\def\trop{\mathrm{trop}}

\def\vir{^{\mathrm{vir}}}
\def\ev{{\mathrm{ev}}}

\definecolor{lightgreen}{HTML}{E8F9E2}
\definecolor{lightblue}{HTML}{ADD8E6}
\definecolor{lightcyan}{HTML}{E0FFFF}
\definecolor{lightred}{HTML}{ffcccb}
\definecolor{lightpink}{HTML}{ffe6e6}
\definecolor{coralpink}{HTML}{F88379}
\definecolor{forestgreen}{HTML}{228B22}
\definecolor{skyblue}{HTML}{448EE4}
\definecolor{skyblue}{HTML}{448EE4}
\definecolor{darkgreen}{HTML}{006400}
\definecolor{darkblue}{HTML}{00008B}
\definecolor{darkorchid}{HTML}{9932CC}

\makeatletter
\def\blfootnote{\xdef\@thefnmark{}\@footnotetext}
\makeatother

\usepackage{setspace}
\usepackage{array}
\usepackage{ifthen}
\usepackage{subcaption}

\usetikzlibrary{perspective}

\usetikzlibrary{plotmarks}

\graphicspath{ {tropgraphs/} }

\title{Elliptic curve counting in toric threefolds: virtual, enumerative, and tropical}
\author{Sae Koyama}
\date{}

\newcommand{\Ov}{\operatorname{Ov}}

\newcommand{\DR}{\operatorname{DR}}
\newcommand{\TC}{\operatorname{TC}}
\newcommand{\Ev}{\operatorname{Ev}}
\newcommand{\vdim}{\operatorname{vdim}}
\newcommand{\Mspace}{\mathcal{M}_{\Lambda}}
\newcommand{\Wspace}{\mathcal{W}_{\Lambda}}
\newcommand{\Winv}{W_{X,\beta}(\underline{\varphi})}

\newcommand{\Linv}{L_{X,\beta}(\underline{\varphi})}
\newcommand{\rig}{{\operatorname{rig}}}
\newcommand{\rub}{{\operatorname{rub}}}

\begin{document}
	
	\maketitle
	
	\begin{abstract}
		
		We study the enumerative geometry of elliptic curves in toric threefolds. We consider enumerative integer invariants, called well-spaced counts, which can be studied using well-spaced genus-one tropical curves in $\mathbb{R}^3$. By comparing this with the logarithmic degeneration formula, we obtain an explicit relationship between logarithmic virtual invariants and these geometric invariants. The result is a logarithmic analogue of a formula of Getzler--Pandharipande for elliptic curves in $\mathbb{P}^3$. As an application, we show that the virtual logarithmic invariants for $\mathbb{P}^3$ with respect to its toric boundary are strictly less than the ordinary Gromov–Witten invariants once the degree is sufficiently large. Several examples are included.

	\end{abstract}
	
	\tableofcontents

	\section{Introduction}
	
	Tropical methods have been a powerful tool in a variety of enumerative contexts. Planar curve counting is solved in all genus \cite{BM08, CMR23, M05}, and genus $0$ counts are known in all dimension \cite{NS06}. One generalization of these results is the celebrated development of \emph{logarithmic Gromov--Witten invariants} \cite{AC14, GS13}. Degeneration techniques have been used to express these invariants as sums indexed by tropical curves \cite{ACGS20}. However, logarithmic Gromov--Witten invariants are not enumerative. On the enumerative side, moving beyond genus $0$ or low dimension requires an understanding of \emph{which} tropical curves contribute to an enumerative count, and \emph{what} the contributions are. The former is solved by Speyer's \emph{well-spacedness} condition \cite{Sp07}, while the latter is solved in work with Cela \cite{CK26}, where we gave a complete answer for genus one curves in toric varieties of any dimension, finding explicit combinatorial contributions for the tropical curves that appear. 
	
	In this work, we compare the enumerative invariants to the logarithmic Gromov--Witten invariants, giving an explicit comparison and showcasing a delicate interplay between logarithmic Gromov--Witten invariants, higher double ramification cycles, tropical curves, and Speyer's well-spacedness condition \cite{ACGS20, HMPPS25, RK24, RSW19II}.

	Let $X$ be a smooth toric threefold, and $\beta$ a curve class on $X$. In this paper, we consider:
	\begin{center}
		(Q) How many genus $1$ curves of class $\beta$ in $X$ pass through fixed curve and point conditions?    
	\end{center}
	
	We are interested in both the enumerative count and the virtual count. 
	
	\subsection{A logarithmic Getzler--Pandharipande relation}

	The Gromov--Witten theory of curves in $\mathbb{P}^3$ was studied by Getzler and Pandharipande, and they found a precise relation between the virtual and enumerative counts \cite[Theorem 6.1]{Ge97}. The present paper may be regarded as a logarithmic analogue of this result. 
	
	Let the curves defining the first $a$ conditions be dual to cohomology classes $\varphi_1, \ldots, \varphi_a$ in $H^*(X)$. Let cohomology classes $\varphi_{a+1}, \ldots, \varphi_{a+b}$ be dual to points, and denote by $\underline{\varphi}$ the collection $(\varphi_1, \ldots, \varphi_{a+b})$. An integer invariant $\Winv$ is defined in \cite{CK26}, which counts (when such a count exists) genus 1 curves of class $\beta$ passing through the fixed curve and point conditions in $X$, where we have additionally labelled the intersection points with the boundary. 
	
	Our main result is a formula relating $\Winv$ to the corresponding logarithmic Gromov--Witten invariant, denoted $\Linv$. The space of logarithmic stable maps is naturally stratified, with components indexed by \emph{tropical types} \cite{ACGS20} (see Section \ref{sec:tropexpan} for definitions). As a consequence, the logarithmic invariants decompose into a sum indexed by certain tropical types meeting tropical point and line conditions. The theory of expanded degenerations then allows the contributions of each tropical type to be reduced to vertex contributions \cite{R19}. More delicate contributions come from superabundant tropical curves, whose deformation space is larger than expected.

	It is shown in \cite{CK26} that there is a tropical correspondence theorem for $\Winv$, which expresses these invariants as a sum over ``well-spaced'' tropical curves. In this work, we the use the geometry of the relevant moduli spaces to obtain a comparison between the enumerative and logarithmic Gromov--Witten invariants.  This is summarised in Table \ref{tab:summarycontributions}. 
	
	\begin{table}[h]
		\centering
		\begin{tabular}{m{7em}||m{6em}|m{6em}|m{5em}|m{13em}} 
			\toprule
			&Well-spaced, rigid & $e(\gamma) = 1$, not rigid & $e(\gamma) = 2$, not rigid & $e(\gamma) =3$ \\
			\midrule
			$W$ contribution & $M_\gamma \cdot w_\gamma$ & $M_\gamma \cdot w_\gamma$ & $M_{\gamma} \cdot w_\gamma$ & No curves \\
			$L$ contribution & $M_\gamma \cdot w_\gamma$ & 0 & 0 & $-\frac{1}{24} |u_1 \wedge u_2|^2 M_\gamma$  \\
			\bottomrule	
		\end{tabular}
		\caption{Contributions to $\Winv$ and $\Linv$ for each tropical type $\gamma$. }
		\label{tab:summarycontributions}
	\end{table}
	
	The integer $e(\gamma)$ is the excess number of dimensions of the deformation space. The multiplicities $M_\gamma$ and $w_\gamma$ are defined in Sections \ref{sec:decomposition} and Section \ref{sec:winv} respectively. The multiplicity $M_\gamma$ is the \emph{total tropical multiplicity} associated to tropical type $\gamma$, which can be expressed as the determinant of an appropriate tropical evaluation morphism. The \emph{weight} of a tropical type $\gamma$ is the number of logarithmic curves which have type $\gamma$. In the $e(\gamma) = 3$ column, we show that each tropical curve that appear is trivalent and has a single genus $1$ vertex $v$, and the vectors $u_1, u_2$  are the direction vectors of any two half-edges adjacent to $v$. 

	Every genus $0$ tropical type with a choice of vertex gives a genus $1$ tropical type with $e(\gamma) = 3$. These contribute to the logarithmic Gromov--Witten count but not the well-spaced count. Conversely, there are tropical types that contribute to the well-spaced count but not the logarithmic Gromov--Witten count. 
	
	Consequently, we obtain the following result. 
	
	\begin{customthm}{A}\label{thmA}
		Let $X$ be a toric threefold, and $\beta$ a curve class on $X$. Fix cohomology classes $\varphi_1, \ldots, \varphi_{a+1} \in H^4(X)$, and $\varphi_a, \ldots, \varphi_{a+b} \in H^6(X)$. Then we have
		$$\Winv = \Linv + E_1 + \frac{1}{24} E_0$$
		where $E_1$ is the number of non-rigid, well-spaced tropical curves counted with multiplicity $M_\gamma\cdot w_{\gamma}$, and $E_0$ is given as the following weighted sum over genus $0$ tropical curves going through the same conditions:
		$$E_0 = \sum_{\gamma}\sum_{v \in V(\gamma)} |u_1 \wedge u_2|^2M_\gamma$$
		where $u_1, u_2$ are the directions of any two edges leaving the vertex $v$. 
	\end{customthm}
	
	 The logarithmic Gromov--Witten invariants can be computed using double ramification cycle and degeneration techniques \cite{ACGS20, HMPPS25, MR24}. The correction terms can be accessed as follows:
	 \begin{enumerate}
	 	\item The correction term $E_0$ is a weighted sum over genus 0 curves. When $X = \mathbb{P}^3$, we can modify the floor diagram methods of Brugall{\'e}--Mikhalkin \cite{BM08} to give a recursive calculation method for them. This is done in Section \ref{sec:floor}. 
	 	\item The correction term $E_1$ has geometric meaning as curves whose circuit components map into a plane.
	 \end{enumerate}
	 This gives new ways of accessing $\Winv$. We give examples of explicit computations and outline general methods in Section \ref{sec:examples}. 
	
	\subsection{Ordinary vs. relative Gromov--Witten invariants}
	
	Consider when $X = \mathbb{P}^3$. We can choose our conditions to be $a$ lines and $b$ points (all other conditions can be expressed as linear combinations of these). We require $a+2b = 4d$, and we denote the well-spaced and logarithmic invariants $W_{a,b}$, $L_{a,b}$ respectively. In this case, the count of \emph{smooth} genus 1 curves meeting the conditions is well-defined, and we denote this number by $n_{a,b}$. We have $n_{a,b} = (d!)^4 W_{a,b}$ (see Remark \ref{rem: P3 enumerative}).

	We compare the ordinary Gromov--Witten theory of $\mathbb{P}^3$ to the logarithmic Gromov--Witten theory of the pair $(\mathbb{P}^3, \partial \mathbb{P}^3)$, where $\partial \mathbb{P}^3$ is the toric boundary. For the latter, we prescribe transverse tangency to the four hyperplanes. In genus $0$, the ordinary and logarithmic Gromov--Witten invariants are both equal to the enumerative count. In genus $1$, the associated moduli spaces share a birational main component, but have different boundary structure. A relation between the two can be found by combining Theorem \ref{thmA} with a formula due to Getzler--Pandharipande \cite{Ge97}. Denoting the ordinary Gromov--Witten invariants $GW^0_{a,b}, GW^1_{a,b}$ for genus 0 and 1 respectively, they state that
	$$n_{a,b} = GW^1_{a,b} + \frac{2d-1}{12}GW^0_{a,b}.$$
	
	The second term of the right hand side can be rewritten as a
	$$\frac{1}{24}\sum_{\gamma}\sum_{v \in V(\gamma)} M_\gamma,$$ 
	where the sum is over genus $0$ tropical types $\gamma$, and $V(\gamma)$ is the vertex set of the underlying graph.
	Hence, we obtain the following. 
	
	\begin{customcor}{B} \label{cor:logandordinary}
		We have an inequality
		$$(d!)^4 GW_{a,b}^1  \geq L_{a,b},$$
		which is strict for $a,b$ large. 
	\end{customcor}

	Values of $a,b$ for which the inequality is strict is given in Section \ref{sec:examples}. In particular, the inequality is strict for degree $d \geq 3$ for all line conditions, or $d \geq 4$ for all point conditions. We remark that $GW_{a,b}^1$ and $L_{a,b}$ can both be negative rational numbers. 

	\begin{remark}
		Mandel and Ruddat show in \cite[Theorem 5.1]{MR20} and \cite[Corollary 5.6]{MR20}, that the ordinary and logarithmic Gromov--Witten invariants in genus 0 are equal after accounting for factor coming from permuting the boundary markings, so the discrepancy observed in Corollary \ref{cor:logandordinary} is a genuine genus 1 phenomenon. In genus 0, they also show that for $X$ a product of projective spaces, if  $D \cdot \beta \leq 1$ for every boundary divisor $D$ of $X$, then the logarithmic and orindary Gromov--Witten invariants agree \cite[Theorem 5.4]{MR20}. The inequality of Corollary \ref{cor:logandordinary} being strict only for degree $d \geq 3$ is evidence that this holds more generally. 
	\end{remark}

	\subsection{Well-spacedness} 
	The enumerative invariants are studied using well-spaced curves. A genus $1$ tropical curve in $\mathbb{R}^3$ is \emph{well-spaced} if for every plane $H$ containing the circuit, we have
	\begin{enumerate}[(i)]
		\item the neighbourhood of the circuit is not contained in $H$; or
		\item at the minimum distance from the circuit to where the curve leaves the plane $H$, the curve leaves $H$ via at least three edges\footnote{Note for a trivalent curve, this is the condition that the minimum distance to where the tropical curve leaves the plane is met twice.}. 
	\end{enumerate} 
	
		\begin{figure}[h]
	\centering
	\begin{tikzpicture}[3d view]
		\draw (0, 0, 0) -- (10, 0, 0) -- (10,10,0) -- (0,10,0) -- cycle;
		
		\draw (3,3,0) -- (5,3,0) -- (7,5,0)--(7,7,0)--(5,7,0)--(3,5,0)-- cycle;
		\draw[thick] (3,3,0) -- (1,1,0);
		\draw[thick] (5,3,0) -- (5,1,0);
		\draw[thick] (7,5,0) -- (9,5,0);
		\draw[thick] (7,7,0) -- (9,9,0);
		\draw[thick] (5,7,0) -- (5,9,0);
		\draw[thick] (3,5,0) -- (1,5,0);
		\draw (0,0,-1) -- (1,1,0) -- (0,0,1);
		\draw (5,0,-1) -- (5,1,0) -- (5,0,1);
		\draw (10,5,-1) -- (9,5,0) -- (10,5,1);
		\draw (10,10,-1) -- (9,9,0) -- (10,10,1);
		\draw (5,10,-1) -- (5,9,0) -- (5,10,1);
		\draw (0,5,-1) -- (1,5,0) -- (0,5,1);
		
		\node at (2.5,2,0) {$\ell_1$};
		\node at (5.5,2,0) {$\ell_2$};
		\node at (8,5.7,0) {$\ell_3$};
		\node at (8,8.5,0) {$\ell_4$};
		\node at (5.5,8,0) {$\ell_5$};
		\node at (2,5.7,0) {$\ell_6$};
	\end{tikzpicture}
	\caption{The tropical curve is well-spaced if and only if the minimum of $\ell_1, \ldots, \ell_6$ is achieved twice.}
	\label{fig:well-spacedness}
	\end{figure}
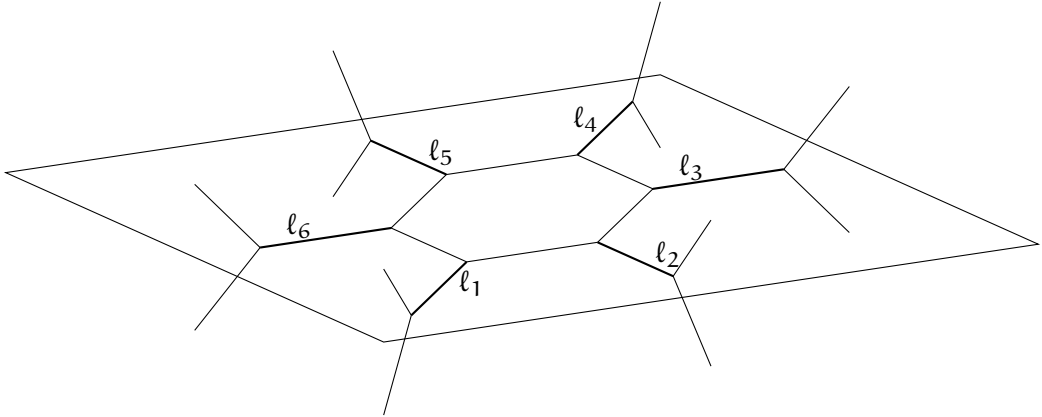
	
	There is a normal crossings compactification of $\cM_{1,n}(\mathbb{P}^n, d)$, the moduli space of maps from smooth $n$-marked genus $1$ curves to $\mathbb{P}^n$, constructed by Vakil and Zinger \cite{VZ08}. It has a modular interpretation as the space of \emph{well-spaced, centrally-aligned} logarithmic stable maps \cite{RSW19}. A similar construction gives $\mathcal{W}_\Lambda(X)$, a logarithmically smooth compactification of the moduli space of genus 1 logarithmic stable maps to $X$ with fixed contact orders, given as the moduli space of \emph{well-spaced, radially-aligned} logarithmic curves \cite{RSW19II}. 
	
	A radially aligned logarithmic curve is well-spaced if and only if its tropicalization is well-spaced. Hence the boundary of $\mathcal{W}_\Lambda(X)$ has strata indexed by well-spaced tropical curves. A key observation for the comparison in Theorem \ref{thmA} is that the moduli space of stable logarithmic curves and well-spaced curves share strata indexed by non-superabundant types. We study the differences for non-superabundant types.

	\subsection{Motivation and prior work}

	Since Mikhalkin's original work on planar rational curve counts \cite{M05}, there has been interest in extending his tropical correspondence theorem to the higher genus and higher dimensional case. The key issue is superabundancy and non-realizability -- that is, the existence of tropical curves which do not arise as the tropicalization of a family of algebraic curves with smooth general fiber. In the elliptic case, Speyer's well-spacedness condition solves the realizability problem \cite{CFPU, KatLift, R16, Sp07}. Speyer poses the question of whether it is possible to use realizability results to produce combinatorial formulas. 
	
	Several advances in recent years have made this question accessible, including the development of logarithmic methods, and a modular interpretation of the well-spacedness condition \cite{RSW19, RSW19II}. The invariants associated to these spaces coincide with the reduced invariants provided by Vakil--Zinger's desingularization of $\overline{\mathcal{M}}_{1,n}(\mathbb{P}^r,d)$, and the logarithmic perspective provides a way to access these tropically. The tropical correspondence theorem of \cite{CK26} provides a full answer to Speyer's question.

	Logarithmic Gromov--Witten invariants are of independent interest. Considerable progress has been made in their study, including degeneration techniques and tropical gluing formulas \cite{ACGS20, ACGS24, R19}, as well as higher double ramification cycles and their enumerative applications \cite{CMS25, HMPPS25, JPPZ17, RK24}. Here we provide a unification of these ideas in the first non-trivial case, giving concrete methods of calculation.  The Getzler-Pandharipande relation of \cite{Ge97} can be viewed as a ``Vakil--Zinger to Gromov--Witten translation", and Theorem \ref{thmA} is a incarnation of this result for all toric threefolds relative to their toric boundary. We emphasise that Theorem \ref{thmA} is independent of the explicit multiplicities found in \cite{CK26}, and we hope that comparisons between invariants, combined with the known multiplicities, can provide new insights into both theories.

	Related work is the study of relative reduced Gromov--Witten invariants for $(\mathbb{P}^3|H)$ by Battistella--Nabijou--Ranganathan \cite{BND19}. While the tropical objects we consider are more complex, considering the full toric boundary makes the correspondence simpler.

	\subsection{Further work} 
	
	A natural direction of exploration is the study of characteristic numbers, which count curves with the condition of tangency to a general hypersurface. This has been studied tropically in the genus 0 case \cite{BBM14}.
	Some related invariants are the descendent invariants, which are generally not enumerative, but have a concrete relation to characteristic numbers in the planar genus $1$ and $2$ case via \emph{modified} $\psi$ classes \cite{GKP02}. There are tropical correspondences for descendent invariants, where the tropical curves being counted are required to have vertices with valence higher than $3$ \cite{BS19, Bo21, KQU24, MR20}. Refined floor diagrams have been used to study these in the planar case \cite{Br15, DH25}, and we expect that similar analysis of the contributing tropical types may give more geometric meaning to these invariants in higher dimension. We remark that when $\psi$-classes are present, even in genus 0 there is a discrepancy between the ordinary and logarithmic invariants \cite{MR20}. Such discrepancies would likely also contribute in a genus 1 setting.

	The aim of this paper is to create a theoretical framework for calculations, and there is scope for better methods or explicit recursion relations. For toric varieties with good properties, it is predicted that Brugall{\'e}--Mikhalkin's generalized floor diagrams can be modified to achieve this. However, the unexpected deformations of superabundant curves create problems when finding constrains that force any tropical curve meeting them to be ``floor decomposed''. Even when the relevant curves of higher genus can be found using floor diagrams, it is not clear how to obtain appropriate multiplicities in a way that is reasonable to compute. Many tropical proofs of algebraic results have used tropical enumeration techniques in a serious way (for instance, see \cite{DH25} or \cite{KM09}) and an explicit calculation framework for higher genus tropical curves is likely to give access to more powerful algebraic theorems.

	\subsection{Notation and conventions} We work over an algebraically closed field with characteristic $0$. A toric variety $X$ has a fan $\Sigma_X$ in cocharacter lattice $N$, whose dual is the character lattice $M$. For $\sigma$ a cone of $\Sigma_X$, the lattice $N_{\sigma}$ is generated by $\sigma \cap N$.

	\subsection{Outline} Background on tropical curves and expansions are give in Section \ref{sec:tropexpan}, along with necessary modifications for calculations. Background on well-spacedness is given in Section \ref{sec:well-spaced}. The contributions to the curve counting invariants are given in Section \ref{sec:calculations}, providing a proof of Theorem \ref{thmA}. Calculation techniques and examples are given in Section \ref{sec:examples}, also giving a proof of Corollary \ref{cor:logandordinary}.

	\subsection{Acknowledgements} I thank my PhD advisor, Dhruv Ranganathan, for invaluable advice and direction. Many thanks to Alessio Cela, whose conversations and feedback have been essential. Thank you also to everyone in the group for all the helpful advice, and for making this journey joyful.
	
	The author's doctorial programme was funded by St John's College, Cambridge. 
	
	\section{Tropical curves and expansions} \label{sec:tropexpan}
	
	\subsection{Tropical curves} \label{sec:tropical} 

	We review some basic definitions on tropical curves. 
	
	\begin{definition}
		A graph $G$ consists of the following data: 
		\begin{enumerate}
			\item (the set of vertices) a finite non-empty set $ V(G)$ ; 
			\item (the set of half-edges) a finite set $ L(G)$ ; 
			\item (the edges) an involution $\iota: L(G) \rightarrow L(G)$; and 
			\item (adjacency) a partition of $L(G)$ indexed by $V(G)$ - that is  $L(G)_V \subset L(G)$ such that $L(G) = \cup_V L(G)_V$ and $L_V \cap L_W = \emptyset$ for vertices $V \neq W$ . 
		\end{enumerate}
	\end{definition}
	
	A pair of distinct elements of $L$ interchanged by $\iota$ is called a \emph{edge} of the graph. Denote the set of edges by $E(G)$. An element of $L$ mapped to itself under $\iota$ is called a \emph{infinite end}. Denote the set of infinite ends by $L^e(G)$. 
	
	\begin{definition} \label{def:tropcurve}
		A \emph{tropical curve} $\Gamma = (G, g, m, \ell)$ consists of the following data: a finite, connected graph $G$, a genus function $g \rightarrow V(G)$, a bijective \emph{marking function} $m:  \{1, \ldots N\} \rightarrow L^e(G)$ to the set of infinite ends, and a \emph{length function} $\ell:E(G) \rightarrow \RR_{>0}$ from the edge set of $G$.

	\end{definition}
	
	The \emph{genus} of a tropical curve $\Gamma$ is 
	$$g(\Gamma) = h_1(G) + \sum_{v \in V} g(v).$$
	The curve $\Gamma$ can be given the structure of a metric space by giving each edge $e$ of $G$ the positive real edge-length $\ell(e)$ to form $G^{\text{met}}$, then defining
	
	$$\Gamma^{\text{met}}=(G^{\text{met}} \amalg \coprod_{i \in \{1, \ldots N\} }\RR_{\geq0}) / \sim$$
	where the equivalence relation glues each $0$ to the vertex which the infinite edge $m(i)$ is adjacent to. The edges of infinite length corresponding to markings are called \emph{legs} or \emph{infinite ends}. We will often refer to $\Gamma^{\text{met}}$ as $\Gamma$.

	\begin{definition}
		A \emph{tropical map from a tropical curve $\Gamma$ to $\RR^r$} is a piecewise linear function $h$ from $\Gamma^{\text{met}}$ to $\RR^r$ with integer slopes, satisfying the following balancing condition: at every point in $\Gamma$, the sum of the directional derivatives of $\varphi$ is 0. 
		
	\end{definition}
	
	\begin{definition}
	A \emph{tropical map from a tropical curve $\Gamma$ to $\Sigma_X$} is a tropical map from $\RR^r$ with the added condition that edges map into cones of $\Sigma_X$. 
	\end{definition}
	
	We also impose a degree condition. Let $\triangle$ be a $r \times N$ matrix whose rows sum to $0$ and let $\triangle_i$ be the vector given by the $i$-th column. We require that the $i$th infinite end has direction vector $\triangle_i$. Call $\triangle$ the \emph{degree} of the tropical map. 
	
	We call $(h, \Gamma)$ a \emph{parametrised tropical curve} of genus $g$ and degree $\triangle$.

	\begin{example} \label{rem:degree1}
		We can take $X = \mathbb{P}^3$, and the degree matrix $\triangle$ to be the \emph{standard degree matrix of degree $d$}, that is the matrix
		$$\begin{pmatrix}
			1 & \cdots & 1 & 0 &  \cdots &0 & \cdots & -1 & \cdots & -1 \\
			0 & \cdots  & 0 & 1 &\cdots & 1 & \cdots & -1 & \cdots & -1 \\
			\vdots & \ddots & \vdots & \vdots& \ddots & \vdots & \ddots & \vdots & \ddots & \vdots  \\
			0 & \cdots & 0 & 0 & \cdots & 0 & \cdots & -1 & \cdots & -1 \\
		\end{pmatrix}$$
		where the columns of the matrix consists of $d$ copies of $e_i$ for each $i$, and $d$ copies of $(-1, \ldots, -1)$. Then we recover Mikhalkin's definition of a tropical curve in $\mathbb{R}^r$. It is conventional to add $n$ columns of $0$s, and the corresponding markings will be used to impose conditions. 
	\end{example}
	
	\begin{remark}
		For a tropical curve $\Gamma$ with underlying graph $G$, we denote by $E(\Gamma)$, $V(\Gamma)$, and $L(\Gamma)$, the edge set, vertex set, and set of half-edges of $G$. 
	\end{remark}
	
	\begin{remark}
		Given a parametized tropical curve $(h, \Gamma)$, each half-edge $\zeta$ in $L(\Gamma)$ has a well-defined edge direction $u_{\zeta}$. 
	\end{remark}

	\begin{definition}
		Fix a lattice $N$. For $u \in N$, the \emph{weight} of $u$ is the maximum non-negative integer $a$ such that 
		$u = au'$
		for $u' \in N$. 
	\end{definition}

	The \emph{combinatorial type} of a tropical curve consists of the data obtained by dropping the edge lengths. That is
	\begin{enumerate}[(i)]
		\item the finite graph, genus, and markings $(G,g,m)$ underlying $\Gamma$; 
		\item for each half-edge $\zeta$ of $G$, the edge direction $u_{\zeta} \in \RR^r$; and 
		\item for each vertex and edge of $\Gamma$, the cone in $\Sigma_X$ which contains it. 
	\end{enumerate}
	We often denote a combinatorial type by $\gamma$.

	The moduli space of tropical maps is a cone complex where each cone is labelled by combinatorial type. The \emph{overvalence} of a graph $G$ is 
	$$\Ov (G) = \sum_{v \in G, \operatorname{val}(v) \geq 3} \operatorname{val}(v) - 3.$$
	The expected dimension of a cone in this moduli space corresponding to combinatorial $\gamma$ is $$N+(r-3)(1-g)- \Ov(G),$$ 
	where $G$ is the underlying graph of $\gamma$. However there are cones in this moduli space that have larger than expected dimension. 
	
	\begin{definition}
		A parametrised tropical curve $(h,\Gamma)$ is \emph{superabundant} if its deformation space has dimension greater than 
		$$N+(r-3)(1-g) - \Ov(G),$$ 
		where $G$ is the underlying graph of $\Gamma$. Otherwise it is called \emph{non-superabundant}. 
		
	\end{definition}

	The excess dimension of the deformation space is denoted $e(h)$. For genus $1$ curves, the only type of superabundancy is \emph{planar superabundance}, where the cycle $h(\Gamma_0)$ is contained in a hyperplane of $\RR^r$. For $r =3$, the tropical curve $(h, \Gamma)$ has excess dimension $1$ if it is contained in a plane but not a line, excess dimension $2$ if it is contained in a line but not contracted, and excess dimension $3$ if the whole cycle is contracted.

	\subsection{Tropicalization of logarithmic spaces} \label{sec:trop}
	
	We refer the reader to \cite[Section 2]{ACGS20}, \cite[Section 1]{MR24}, \cite[Section 4]{U17} for an introduction to generalized cone complexes and Artin fans. We summarise here. 
	
	To a logarithmic scheme $X$, with logarithmic structure in Zariski topology, we can associate to it a generalized cone complex $\Sigma_X$ (also denoted $\Sigma(X)$). This is built out of cones 
	$$\sigma_x = \Hom(\overline{\cM}_{X, x}, \RR_{\geq 0}) \subset (N_x)_{\RR} = \Hom(\overline{\cM}_{X, x}, \ZZ)_{\RR}$$
	for $x \in X$, and glued by the dual of generalization maps $\overline{\cM}_{X, x} \rightarrow \overline{\cM}_{X, x'}$. The association is functorial: a logarithmic morphism $f: X \rightarrow Y$ induces a map of generalized cone complexes $\Sigma(f): \Sigma_X \rightarrow \Sigma_Y$.

	The combinatorial structure of the cone complex $\Sigma_X$ is given by the associated \emph{Artin fan} (an Artin stack logarithimcally \'etale over $\spec k$), denoted $\mathcal{A}_X$. It is constructed as follows. For $\sigma$ a cone $\sigma \subset N_\mathbb{R}$ in $\Sigma_X$, with corresponding monoid $P = \sigma^\vee \cap N^\vee$. Let
	$$\cA_\sigma = [\spec k[P]/\spec k[P^{\text{gp}}]]$$
	be the stack quotient, which we call the \emph{Artin cone} of $\sigma$. The Artin fan $\cA$ is given by gluing these together, or more succinctly, as a colimit 
	$$\cA_X = \cA_{\Sigma_X} = \varinjlim_{\sigma \in \Sigma_X} \cA_\sigma.$$
	
	There is a natural morphism $X \rightarrow \cA_X$. Given $X \rightarrow Y$ a sufficiently nice logarithmic morphism, we have:
	\begin{center}
		\begin{tikzcd}
			X \arrow[d] \arrow[r]   & Y \arrow[d]   \\
			\mathcal{A}_X \arrow[r] & \mathcal{A}_Y
		\end{tikzcd} 
	\end{center}
	This holds, for instance, if $X$ is logarithmically smooth \cite[Proposition 2.8]{ACGS20}.

	\begin{example}
		\begin{enumerate}
		\item Suppose $(X, \partial X)$ is a simple normal crossings pair, with $\partial X = D_1 \cup \ldots \cup D_k$. Then $X$ with the divisorial logarithmic structure given by $\partial X$ has cone complex $\Sigma_X$ is constructed from the combinatorial data of the intersections of $D_i$. For each $D_i$, there is a corresponding ray $\rho_i$ in $\Sigma_X$. For each non-empty $\ell$-fold intersection $D_{i_1} \cap \ldots \cap D_{i_\ell}$, there is an associated $\ell$-simplex spanned by the rays $\rho_{i_1}, \ldots, \rho_{i_\ell}$. 
		
		\item Suppose $X$ is a toric variety with toric boundary $\partial X$. Then the associated cone complex $\Sigma_X$ is the fan of $X$, but forgetting its embedding into the character lattice $N$. We will often use $\Sigma_X$ both for the fan as a toric variety and its fan as a logarithmic scheme. 
		\end{enumerate}
	\end{example}

	\begin{definition}
		A \emph{subdivision} $\tilde{\Sigma}$ of a cone complex $\Sigma$ is a morphism of cone complexes that is bijective on supports, such that the lattice points of each cone $\sigma \in \tilde{\Sigma}$ are exactly the intersection of the lattice points of $\Sigma$ with the image of $\sigma$. 
	\end{definition}

	A subdivision of a cone complex $\Sigma_X$ associated to $X$ corresponds to a modification of Artin fans $\tilde{\mathcal{A}}_X \rightarrow \mathcal{A}_X$. This can be pulled back\footnote{In the category of fine, saturated logarithmic stacks.} along $X \rightarrow \mathcal{A}_X$ to give a logarithmic modification $\tilde{X} \rightarrow X$.

	\subsection{Logarithmic curves with expanded targets} We review the theory of logarithmic Gromov--Witten theory with expansions \cite{R19}. 
	
	Consider a toric variety $X$ with its toric boundary $\partial X = D_1 \cup \ldots \cup D_k$. Fix a curve class $\beta \in H_2(X)$, integers $g, n \geq 0$, and a $k \times m$ matrix $(c_{ij})$ with integer entries. This data will be denoted $\Lambda$. We wish to study maps of pairs
	$$f: (C, p_1, \ldots, p_n, q_1, \ldots, q_m) \rightarrow (X, \partial X)$$
	where $C$ is a smooth, genus $g$ curve with homology class $\beta$, meeting the boundary components $D_i$ at $q_j$ with contact order $c_{ij}$. There is a Deligne--Mumford stack $\mathcal{M}_\Lambda^\circ$ parametrizing such maps. 
	
	This space is compactified by modifying the target. 
	
	\begin{definition} \cite[Definition 2.2.1]{R19}
		A \emph{tropical expansion} of $\Sigma$ over a base $\Sigma_B $ is a cone complex $\tilde{\Sigma}$ together with a morphism of cone complexes 
		\begin{center}
			\begin{tikzcd}
				\tilde{\Sigma} \arrow[d] \arrow[r, "\pi"] & \Sigma \\
				\Sigma_B                                 &       
			\end{tikzcd}
		\end{center}
		satisfying:
		\begin{enumerate}[(i)]
			\item Every cone of $\tilde{\Sigma}$ surjects onto a cone of $\Sigma_B $. 
			\item The map $\tilde{\Sigma} \rightarrow \Sigma$ is a subdivision of cone complexes. 
			\item Over any point $p \in \Sigma_B $, the map $\pi$ restricts to a polyhedral subdivision $\tilde{\Sigma}_p \rightarrow \Sigma_p$. 
			\item  The fiber over $0 \in \Sigma_B $ maps isomorphically onto a union of faces in $\Sigma$. 
		\end{enumerate}
		
	\end{definition}

	\begin{example}
		Consider $X \times \mathbb{A}^1$, whose tropicalization is $\Sigma(X \times \mathbb{A}^1) = \Sigma_X \times \mathbb{R}_{\geq 0}$. A polyhedral subdivision $\tilde{\Sigma}_X$ of $\Sigma_X$ induces a subdivision of $\Sigma(X \times \mathbb{A}^1)$ (by taking the cone over $\tilde{\Sigma}_X \times \{1\}$), and hence a logarithmic modification of $X \times \mathbb{A}^1$. The fiber over $0$ of the projection $X \times \mathbb{A}^1 \rightarrow \mathbb{A}^1$ is called an \emph{expansion} of $X$. 
	\end{example}

	\begin{definition}\cite[Definition 2.1.11]{R19}
		Given a tropical expansion $\tilde{\Sigma}$ over $\Sigma$ over $\Sigma_B $, the associated \emph{logarithmic expansion} is obtained as the corresponding logarithmic modification $\tilde{\mathcal{X}}$
		\begin{center}
			\begin{tikzcd}
				\tilde{\mathcal{X}} \arrow[d] \arrow[r] & X \times B \arrow[d]                  \\
				\mathcal{A}_{\tilde{\Sigma}} \arrow[r]  & \mathcal{A}_{\Sigma \times \Sigma_B }
			\end{tikzcd}
		\end{center} 
		with its map to the base $B$. 
	\end{definition}
	
	The space $\mathcal{M}_\Lambda$ is constructed as the moduli space of logarithmic curves mapping to expansions of $X$ (with specified curve class, genus, and contact orders):
	\begin{center}
		\begin{tikzcd}
			C \arrow[d] \arrow[r, "f"] & \tilde{\mathcal{X}} \arrow[ld] \\
			S                          &                     
		\end{tikzcd}
	\end{center}
	with a stability condition. 
	
	The moduli space $\mathcal{M}_\Lambda$ admits evalulation morphisms and virtual class, and Gromov--Witten invariants can be defined in the usual way. It was shown in \cite{R19} that these invariants are the same as the usual logarithmic Gromov--Witten invariants as in \cite{GS13}. 
	
	\begin{remark}
		The construction of $\mathcal{M}_\Lambda$ depends on a auxillary polyhedral choice, coming from the fact that we may perform subdivisions to obtain new moduli spaces with the same relevant properties. However, these choices give the same Gromov--Witten invariant and we will not be concerned with the choice used. 
	\end{remark}
	
	We are interested in the case where $X$ is a smooth, proper toric threefold. Fix integers $a,b$ such that 
	$$a+2b= \int_\beta c_1(T_X).$$ 
	Let $\varphi_1, \ldots, \varphi_a \in H^{4}(X)$ and $\varphi_{a+1}, \ldots, \varphi_{a+b} \in H^{6}(X)$ be cohomology classes on $X$. Consider the logarithmic Gromov--Witten invariant
	$$\Linv = \int_{[\mathcal{M}_{\Lambda}(X)]^{\text{vir}}} \ev_{1}^*(\varphi_1) \cdots  \ev_{n}^*(\varphi_{a+b}).$$

	We can tropicalize a family of logarithmic curves to expansions of $X$. 
	\begin{center}
		\begin{tikzcd}
			\Sigma_C \arrow[r, "\Sigma(f)"] \arrow[d, "\Sigma(\pi)"'] & \tilde{\Sigma}_X \arrow[dl] \\
			\Sigma_S .
		\end{tikzcd}
	\end{center}

	The fibres of $\Sigma(\pi)$ are tropical curves, and the tropicalization naturally has an interpretation as a family of parametrized tropical curves. Indeed, they map to polyhedral subdivisions of the fan $\Sigma_{X}$, but since $X$ is a complete toric variety, its fan is complete, and we may consider these as maps to $\mathbb{R}^3$.

	\begin{example} \label{exa:expansions} 
		We will use this example in Section \ref{sec:exdim2}. 
		Consider a tropical type $\gamma$ comprised of two four valent vertices, joined by two edges, with the neighbourhood of the cycle contained in a plane. (For a picture, see right hand side of Figure \ref{fig:twochoices}.) Changing coordinates, we may assume that the edges in the cycle are parallel to the $x$-axis, and that the plane is parallel to the $x$-$y$ plane.
		
		Suppose we have a logarithmic map to expansions $f\colon C \rightarrow \tilde{X}$ which tropicalizes to $\gamma$. Then the curve $C$ is given by gluing two components $C_i \cong \mathbb{P}^1$s at two points, and each component $C_i$ is mapped to $X_i$, where $X_i$ are toric 3-folds glued along a toric boundary divisor. Let this divisor be $D$. Each $\mathbb{P}^1$ is marked with 4 points, which without loss of generality, these marked points are $0,1,\infty,p_i$ respectively, $p_1,p_2 \in \mathbb{C} \setminus \{0,1,\infty\}$, as shown in Figure \ref{fig:nodal}. Here we glue the two $0$s to each other, and the point $p_1$ to the point $p_2$.   
		
		\begin{figure}[h]
			\centering
			\begin{tikzpicture}
				\draw plot [smooth cycle,tension=1] coordinates {(0,0) (2,0) (1,1) (2,2) (0,2)};
				\draw plot [smooth cycle,tension=1] coordinates {(2.1,0) (3.1,1) (2.1,2) (4.1,2) (4.1,0)};
				\fill (0,0) circle[radius=2pt];
				\fill (0,2) circle[radius=2pt];
				\fill (2.05,0.1) circle[radius=2pt];
				\fill (2.05,1.9) circle[radius=2pt];
				\fill (4.1,2) circle[radius=2pt];
				\fill (4.1,0) circle[radius=2pt];
				\node at (0.2, 0.2) {$\infty$};
				\node at (0.2, 1.8) {$1$};
				\node at (1.7, 0.1) {$p_1$};
				\node at (1.7, 1.9) {$0$};
				\node at (2.4, 0.1) {$p_2$};
				\node at (2.4, 1.9) {$0$};
				\node at (3.9,1.8) {$1$};
				\node at (3.9,0.2) {$\infty$};

				\fill (-4,1) circle[radius=2pt];
				\draw (-4, 1.05) -- (-3, 1.05);
				\draw (-4, 0.95) -- (-3, 0.95);
				\draw (-4, 1) -- (-5, 2); 
				\draw (-4, 1) -- (-5, 0); 
				\node at (-5.2,-0.1) {$\infty$};
				\node at (-5.1,2.1) {$1$};
				\node at (-2.8,0.8) {$p_1$};
				\node at (-2.8,1.2) {$0$};
				\node at (-6,2) {$\Gamma_1$};
				
				\fill (8,1) circle[radius=2pt];
				\draw (8, 1.05) -- (7, 1.05);
				\draw (8, 0.95) -- (7, 0.95);
				\draw (8, 1) -- (9, 2); 
				\draw (8, 1) -- (9, 0); 
				\node at (9.2,-0.1) {$\infty$};
				\node at (9.1,2.1) {$1$};
				\node at (6.8,0.8) {$p_2$};
				\node at (6.8,1.2) {$0$};
				\node at (10,2) {$\Gamma_2$};
				
			\end{tikzpicture}
			\caption{The corresponding nodal curve mapping to an expansion.}
			\label{fig:nodal}
		\end{figure}
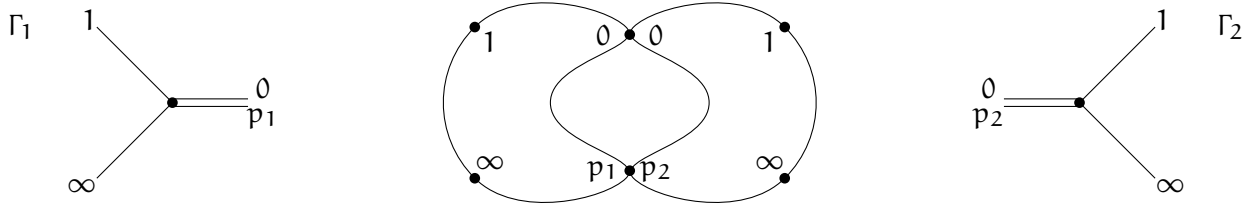	
		
		In the interior of each $X_i$, the map $C_i \rightarrow X_i$ is given by a rational function $\mathbb{P}^1 \setminus{0,1,\infty, p_i} \rightarrow (\mathbb{C}^*)^3$, with orders of zeros and poles at the four marked points determined by the direction vectors of infinite ends of $\Gamma_i$. This in turn is determined by two sets of three rational functions 
		$$\mathbb{P}^1 \setminus\{0,1,\infty, p_i\} \rightarrow \mathbb{C}^*,$$ which we call $f_i,g_i,h_i$, corresponding to directions $e_1,e_2,e_3$ in the character lattice $N$ respectively. These are determined by the edge directions up to constants $a_i,b_i,c_i$. Since we assumed that the tropical curve is contained inside the plane $\langle e_1, e_2 \rangle$, $h_i = c_i$ are constant. 
		
		The gluing is given by imposing that, on the toric boundary divisor $D$, the images of the two $0$s and $p_i$s coincide respectively. Since the two corresponding infinite ends are pointing in direction $e_1$, this reduces to $g_1(0) = g_2(0), g_1(p_1) = g_2(p_2),$ and $c_1 = c_2$. The first equation gives a relation between $b_1$ and $b_2$ (since $g_1(0), g_2(0) \neq 0$). Thus, the second equation imposes a relation between $p_1$ and $p_2$ of the form
		$$\frac{b_1}{b_2}(p_1 -1)^{d_1} =  (p_2 - 1)^{d_2},$$
		with $d_i \in \mathbb{Z}_{> 0}$. For each $p_1$, there are finitely many solutions for $p_2$. The parameters $a_i, b_i, c_i$ are subsumed by the rubber action, and we are left with a 1 dimensional locus of curves parametrized by $p_1$. The excess number of dimensions is 1. 
	\end{example}

	\subsection{Decomposition} \label{sec:decomposition} 
	
	The decomposition theorem \cite[Theorem 1.2]{ACGS20} relates the logarithmic Gromov--Witten invariant of a degenerate space to a sum indexed by tropical types. We prove a variant of this theorem using toric intersection theory. The strategy is as in the proof of the correspondence theorem of \cite{CK26}, but modification is needed because $\mathcal{M}_{\Lambda}$ is not logarithmically smooth. 
	
	\subsubsection{Reminder on toric intersection theory}
	
	We review the theory of Minkowski weights developed in Fulton--Sturmfels \cite{FS97}. Let $X$ be an $r$-dimensional complete toric variety with fan $\Sigma_X$. Given a Chow cohomology class $c \in A^k(X)$, there is a well-defined function from cones of codimension $k$ to $\mathbb{Q}$, given by taking the degree of $c \cap V(\sigma)$. Such functions satisfy the following condition. 
	
	\begin{definition}
		A $\mathbb{Q}$ valued function on the codimension $k$ cones of $\Sigma$ (denoted $\Sigma^{(k)}$) is \emph{balanced} if it satisfies 
		$$\sum_{\sigma \in \Sigma^{(k)}: \sigma \subset \tau} \langle u, n_{\sigma, \tau} \rangle \cdot c(\sigma) = 0 $$
		where $\tau$ is a cone of codimension $k+1$ and $n_{\sigma,\tau}$ is the generator of $N_\sigma/N_\tau$. A balanced function of this form is called a \emph{Minkowki weight}, or a \emph{tropical $(r-k)$-cycle}. 
	\end{definition}
	
	Given two Minkowski weights $c, c'$ on cones of codimension $k$, they can be added in a natural way. If $c, c'$ are Minkowski weights on cones of codimension $p,q$ respectively, we can define a Minkowski weight on cones of codimension $p+q$, as follows
	$$c\cdot c' (\gamma) =  \sum_{(\sigma, \tau) \in \Sigma^{(p)} \times \Sigma^{(q)}} m_{\sigma, \tau}^\gamma c(\sigma)c'(\tau)  $$
	where the sum is over cones such that $\sigma$ meets $\tau+v$, with $v$ a \emph{generic} displacement vector, and $\gamma \subset \sigma, \tau$ such that $\operatorname{codim}(\sigma) + \operatorname{codim}(\tau) = \operatorname{codim}(\gamma)$. Here $m_{\sigma, \tau}^\gamma = [N: N_\sigma + N_\tau]$. This gives the Minkowski weights the structure of a graded ring. 
	
	\begin{theorem}[Fulton--Sturmfels]
		The operational Chow cohomology ring of $X$ is naturally isomorphic to the ring of Minkowski weights on $\Sigma$. 
	\end{theorem}
	
	\subsubsection{Decomposition theorem}
	
	We state and prove our decomposition theorem. The proof follows the method outlined in \cite[Chapter 4.6]{CMR23}, and used in \cite[Theorem B]{CK26}. Modifications are needed since $\mathcal{M}_{\Lambda}$ is not logarithmically smooth. 
	
	\begin{definition}
		We call a collection of $a$ tropical 1-cycles\footnote{We may refer to these as \emph{tropical lines}.} and $b$ points in $\mathbb{R}^3$ an \emph{affine condition}. We denote it $A = (\ell_i, p_j)$. 
	\end{definition}

	\begin{definition}
		We say a tropical curve $(h,\Gamma)$ is \emph{rigid for the affine condition} $A$ if there is no local deformation in $\mathcal{M}_\Lambda^\trop$ which preserves the combinatorial type and continues to meet the affine conditions. 
		
		A combinatorial type $\gamma$ (with at least one tropical curve meeting the conditions) is \emph{rigid} for $A$ if all the tropical curves with combinatorial type $\gamma$ meeting the affine condition $A$ is rigid for $A$.
	\end{definition}
	
	If the affine condition is empty, then we recover the definition of rigidity in \cite[Definition 3.6]{ACGS20}. In this case, the only rigid curves are those with one vertex at the origin and legs along the rays of the fan $\Sigma_X$ of $X$.

	As with $\operatorname{ACGS}_\Lambda(X)$, the space $\mathcal{M}_\Lambda(X)$ is stratified with strata indexed by tropical type:
	$$\mathcal{M}_\Lambda(X) = \bigcup_{\gamma} \mathcal{M}_\gamma.$$
	
	\begin{definition}
		For $\gamma$ a strata of virtual dimension $0$, the \emph{virtual multiplicity} $m_\gamma^{\vir}$ is the degree of the virtual fundamental class on this strata. 
	\end{definition}
	
	The virtual multiplicity $m_\gamma^{\vir}$ is a (potentially negative) rational number. 
	
	Since we assumed that $X$ is a smooth toric variety, cohomology classes in $X$ correspond to tropical cycles (balanced Minkowski weights) \cite{FS97}. Fixing cohomology classes $\varphi_1, \ldots, \varphi_a \in H^{2r-2}(X)$, let the corresponding tropical 1-cycles in $\mathbb{R}^n$ be $\ell_1^{\trop}, \ldots, \ell_a^{\trop}$.

	We define our tropical multiplicity. Suppose $(h, \Gamma)$ is a tropical curve meeting affine conditions $A$. Let $\sigma_\gamma$ be the cone in $\Sigma(\mathcal{M}_\Lambda)$ associated to $\gamma$. Extend the edge in the tropical line which meets the relevant marking, so that it becomes a line $\ell'_i$ in $\mathbb{R}^3$. Quotient out by all such lines, and we obtain a tropical evaluation space $\Ev'$ of dimension $D$. 
	
	We now have a map
	\begin{equation} \label{eqn:ev'}
		\sigma_{\gamma} \rightarrow (\mathbb{R}^3)^a \times (\mathbb{R}^3)^b \rightarrow \Ev' = \prod_{i = 1}^a(\mathbb{R}^3/\ell_i' ) \times (\mathbb{R}^3)^b,
	\end{equation}
	which gives an isomorphism on the associated map of vector spaces. 
	
	\begin{definition}\label{def:totalmult}
		The multiplicity $M_\gamma$ is the determinant of this linear map. 
	\end{definition}

	\begin{proposition}[Decomposition] \label{prop:decomp}
		Let $X$ be a smooth toric variety. Let $a,b$ be non-negative integers. Fix cohomology classes $\varphi_1, \ldots, \varphi_a \in H^{2r-2}(X)$, and let $\varphi_{a+1}= \ldots= \varphi_{a+b} \in H^{2r}(X)$ be dual to a point class. Then the associated logarithmic Gromov--Witten invariant satisfies
		$$\Linv = \sum_\gamma M_\gamma m_\gamma^{\vir},$$
		where the sum is taken over tropical types $\gamma$ which are rigid over the affine condition given by tropical line conditions $v_i + \ell_i^{\trop}$, and point conditions $u_j$. Here $(v_i, u_j)$ is a generic tuple of vectors, and $\ell_i^{\trop}$ is the tropical cycle associted to $\varphi_i$. 
	\end{proposition}

	\begin{proof}
		Consider the evaluation map 
		$$\ev: \mathcal{M}_\Lambda \rightarrow X^n.$$
		We obtain Gromov--Witten invariants by pushing forward the virtual fundamental class and capping it with the chosen cohomology classes. Since $X$ is a smooth toric variety, $X^n$ is also, and the cohomology of $X^n$ can be understood in terms of Minkowski weights. However, the strata of $\mathcal{M}_\Lambda$ do not map nicely to the strata of $X^n$. To understand the pushforward of the virtual fundamental class in terms of strata (and hence tropical types), we modify the map. The proof proceeds in four steps. 
		
		\textbf{1. Perform the subdividision.}
		There exists a subidivison of fans such that the induced modification
		$$\ev^\dagger: \mathcal{M}_\Lambda^\dagger \rightarrow (X^n)^\dagger$$
		has the property that the corresponding morphism of Artin fans is flat. Pulling back cohomology classes along $(X^n)^\dagger \rightarrow X^n$ gives the invariants we are interested in. We can further assume that $(X^n)^\dagger$ is also a smooth toric variety, and that the cones in the tropicalization of $\mathcal{M}_{\Lambda}^\dagger$ are simplicial.

		\textbf{2. Express the Gromov--Witten invariant using Minkowski weights.} The class $\Ev^\dagger_*([\Mspace^\dagger]^{\vir})$ is a class in $A_{N}((X^n)^\dagger)$, where $D = \vdim (\Mspace)$. By Poincar{\'e} duality, this corresponds to a cohomology class $c_{\vir} \in A^{3n - N}((X^n)^\dagger)$, which we interpret as a function from the $D$-dimensional cones to $\mathbb{Q}$. 
		Let $c \in H^{\vdim(\Mspace)}$ be the cohomology class given as the product of $\varphi_1, \ldots, \varphi_{a+b}$, interpreted as a Minkowski weight. Then the logarithmic Gromov--Witten invariant is given as the degree of $c_{\vir} \cdot c$. By \cite[Proposition 3.1]{FS97}, this is given as a sum
		\begin{equation} \label{eq:sum}
			\sum_{(\sigma, \tau) \in \triangle^{(3n - D)} \times \triangle^{(D)}} m_{\sigma, \tau}^{0} \cdot c_{\vir}(\sigma) \cdot c(\tau)
		\end{equation}
		where $m_{\sigma, \tau}^0$ is $[N:N_\sigma+N_\tau]$ if the cone $\sigma$ meets $\tau+v$, and $0$ otherwise. Here $v$ is a generic displacement vector.

		\textbf{3. Interpret (\ref{eq:sum}) as sum over tropical types meeting conditions.} Each cone $D$-dimensional cone $\sigma$ in $\Sigma((X^n)^\dagger)$ corresponds to a cohomology class by Poincar{\'e} duality, which we will call $PD(\sigma)$. Explicitly, this can be expressed as the piecewise polynomial function $x_1 \ldots x_D$ on the cone $\sigma$, and is $0$ on all rays not in $\sigma$. The number $c_{\vir}(\sigma)$ is the cap product of $\ev^\dagger_*([\Mspace]^{\vir})$ with this cohomology class (by associativity of cap product).

		By the construction of the moduli space \cite{R19}, we have the following commutative diagram 
		\begin{center}
			\begin{tikzcd}
				(\Mspace)^\dagger \arrow[r] \arrow[d]    & (X^n)^\dagger \arrow[d]    \\
				\mathcal{A}((\Mspace)^\dagger) \arrow[r, "f"] & \mathcal{A}((X^n)^\dagger).
			\end{tikzcd}
		\end{center}
		
		We can pull back the cohomology class $PD(\sigma)$ along the morphism $f$ between Artin fans. This corresponds to the piecewise polynomial pulled back along the morphism of cones. If $\sigma_1', \ldots, \sigma_t'$ are $D$-dimensional cones in $\Sigma(\Mspace)$ mapping surjectively onto $\sigma$, then the corresponding piecewise polynomial has degree exactly $D$ precisely on these cones $\sigma'_i$, and $0$ elsewhere. We can calculate $c_{\vir}(\sigma)$ by capping the cohomology class in $H^*((\Mspace)^\dagger)$ corresponding to this piecewise polynomial with the virtual fundamental class. 
		
		In particular this is non-zero only if $\sigma$ is in the image of the morphism of cone complexes $\Sigma(\ev^\dagger)$. The sum runs over all cones $\sigma$ in this image that meet $\tau +v$. In other words, over all cones $\sigma$ that are the image of cones in $\Sigma((\Mspace)^\dagger)$ corresponding to rigid tropical types which contain a tropical curve where the evaluation morphism sends the marked points to the given point and line conditions.

		\textbf{4. Study the contributions of each tropical type.} 
		For each $D$-dimensional cone $\sigma'$ in $\Sigma((\Mspace)^\dagger)$ corresponds to a stratum of virtual dimension $0$, say with rays $\rho_1, \ldots, \rho_D$. The multiplicity of $\sigma'$, denoted $\operatorname{mult}(\sigma')$ is the index of the lattice spanned by the rays $\rho_i$ in $N_{\sigma}$. The intersection of the divisors on $\mathcal{A}((\Mspace)^\dagger)$ associated to $\rho_i$ is $(1/\operatorname{mult}(\sigma')) [V(\sigma')]$. 
		
		Let $P$ be the piecewise polynomial function which is the product of the $D$ piecewise linear functions with slope $1$ along $\rho_i$ and $0$ along all other rays. From the above, the cohomology class corresponding to $P$ gives $1/\operatorname{mult}(\sigma')$ when capped with the virtual fundamental class. It follows that
		$$c_{\vir} (\sigma) = \sum m_{\gamma_i} m_{\gamma_i}^{\vir}$$
		where for the tropical type $\gamma_i$ corresponding to $\sigma_i'$, the multiplicity $m_{\gamma_i}$ is the determinant of the linear map between $N_{\sigma_i}$ and $N_{\sigma}$. 
		
		Finally, note that  $m_{\sigma, \tau}^0$ is the lattice index $[N:N_\sigma+N_\tau]$. We have $M_\gamma = m_{\sigma, \tau}^0 m_\gamma$ and the result follows. 
	\end{proof}

	\subsubsection{Relation to degeneration} \label{sec:degenvsstrata}
	
	There are two ways to set up tropical curve counting, either by expanded degenerations, or by degeneration \cite{ACGS20, R19}. For the latter, we form a degeneration of the moduli space by considering a degenerating family of point and line conditions in $X$ (see \cite{Gr23} for a detailed discussion). To access the splitting and gluing formulas developed in the degeneration setting, a relation between the two perspectives is required.

	Given a stratum $\mathcal{M}_\gamma$ in $\Mspace$ associated to tropical curve $\gamma$, a curve in the interior of $\mathcal{M}_\gamma$ maps to a fixed expansion of $X$, which we will call $X_\gamma$. In the boundary of $\mathcal{M}_\gamma$, we have curves mapping to further expansions of $X_\gamma$. Thus $\mathcal{M}_\gamma$ can be considered as the moduli space of logarithmic stable maps to expansions of $X_\gamma$, \emph{up to rubber} in the components arising from expansions of $X$.
	
	A ``rigidification'' of the problem is constructed in \cite[Section 5]{MR25}. We summarise the key results. 
	Let the corresponding cone of $\mathcal{M}_\gamma$ in $\Sigma(\Mspace)$ be $\sigma_\gamma$. Choose a point $q$ in the interior of $\sigma_\gamma$. This is a choice of a fixed parametrised tropical curve $\Gamma$ of type $\gamma$. This induces a polyhedral subdivision of the fan of $X$, hence a degeneration $\mathcal{X}$ of $X$ over $\mathbb{A}^1$, and thus a degeneration $\Mspace(\mathcal{X})$ of $\Mspace$. The moduli space $\mathcal{M}_\gamma^\rig$ of logarithmic stable maps to expansions of $X_\gamma$ with tropical type $\gamma$ appears as a component of $\Mspace(\mathcal{X}_0)$.
	
	Tropically, this corresponds to a stellar subdivision of $\sigma_\gamma \times \mathbb{R}_{\geq 0}$ associated to the ray through $q$. The stratum $\mathcal{M}_\gamma^{\rig}$ corresponds to the ray through $q$, and we have a natural morphism $p:\mathcal{M}_\gamma^\rig \rightarrow \mathcal{M}_\gamma$. It is shown in \cite{MR25} that 
	\begin{enumerate}
		\item the pullback of $[\mathcal{M}_\gamma]^{\vir}$ coincides with $[\mathcal{M}_\gamma^\rig]^{\vir}$; and 
		\item there is a \emph{rigidifying insertion} $\eta$ such that $p_*([\mathcal{M}_\gamma^\rig]^{\vir} \cap \eta) = [\mathcal{M}_\gamma]^{\vir}$. 
	\end{enumerate}

	We will not use the full power of this result, but note down a simple corollary. 
	
	\begin{corollary} \label{cor:vanishing}
		If $[\mathcal{M}_{\gamma}^{\rig}]^{\vir}$ vanishes, then $m_{\gamma}^{\vir} = 0$. 
	\end{corollary}

	\subsection{Splitting and gluing} \label{sec:split}

	So far, we have reduced the problem of calculating logarithmic Gromov--Witten invariants to a tropical enumeration problem and a calculation of $m_\gamma^{\vir}$. To calculate these virtual contributions, we will use the gluing formula developed by Maulik--Ranganathan \cite{MR24, R19}.

	Let $\gamma$ be a rigid tropical type with underlying graph $G$. To each vertex $v \in V(G)$, we have a toric variety $X_v$, with fan given by the star of $v$ in $G$. We denote by $\mathcal{M}_v$ the moduli space of maps to expansions of $X_v$, with degree and genus given by the star of $v$. There is a \emph{cutting morphism}
	$$\mathcal{M}_\gamma^\rig(X) \rightarrow \prod_v \mathcal{M}_v.$$
	
	On the interior, each edge $e$ adjacent to $v$ defines a toric divisor $D_e \subseteq X_v$, and a natural evaluation morphism $\mathcal{M}_v^\circ \rightarrow D_e^\circ$, and hence
	$$\prod_v \mathcal{M}_v^\circ \rightarrow \big(\prod_e D_e^\circ \big)^2$$
	By a suitable modification of $\prod_v \mathcal{M}_v$ and $\prod_e D_e^2$, this may be extended to a morphism
	$$\bigtimes_v \mathcal{M}_v \rightarrow \big( \bigtimes_e \cD_e \big)^2,$$
	where one of the evaluation morphisms to $ \bigtimes_e \cD_e$ is combinatorially flat. 
	
	\begin{remark}
		Transversality conditions force modifications $\cD_e$ of $D_e$ (see \cite[Example 5.4.2]{R19}). At the boundary of the moduli space, components are never allowed to ``fall into the boundary'' of $D_e$. Instead, a component bubbles off, maintaining transversality. 
	\end{remark}
	
	\begin{theorem} \cite[Theorem 8.3.2]{MR24}\label{thm:split}
		Form the following fiber square
		\begin{center}
			\begin{tikzcd}
				\mathcal{G}_\gamma \arrow[d] \arrow[r]               & \bigtimes_{v \in V(G)} \mathcal{M}_v \arrow[d] \\
				 \bigtimes_e \cD_e \arrow[r, "\triangle"] & \big(\bigtimes_e \cD_e \big)^2                        
			\end{tikzcd}
		\end{center}
	where $\triangle$ is the diagonal morphism. We have that
	\begin{enumerate}
		\item there is a morphism $\mu: \mathcal{M}^\rig_\gamma \rightarrow \mathcal{G}_\gamma$ which is finite and {\'e}tale of degree $\frac{1}{\ell(\gamma)}\prod w_e$. Here $w_e$ is the expansion factor on the edge $e$ of the rigid tropical curve. The factor $\ell(\gamma)$ is the product, taken over the edges $e'$ of the image of $\gamma$, of the least common multiple of the weights of the direction vectors of edges which map to $e'$; 
		\item $[G_\gamma]^{\vir} = \triangle^![\bigtimes_{v \in V(G)} M_v]^{\vir}$; and 
		\item $\mu_*[\mathcal{M}_\gamma^\rig]^{\vir} = \frac{1}{\ell(\gamma)} \prod_e w_e[\mathcal{G}_\gamma]^{\vir}$. 
	\end{enumerate}
	\end{theorem}

	\section{Well-spacedness} \label{sec:well-spaced}
	
	\subsection{Tropical well-spacedness}
	
	\begin{definition} \label{def:wellspaced}
		Let $h: \Gamma \rightarrow \RR^3$ be a parametrized tropical curve of genus 1, with cycle $\Gamma_0$. Then we say that $(h, \Gamma)$ is \emph{(tropically) well-spaced} if for every plane $H$ containing $h(\Gamma_0)$, either:
		\begin{enumerate}[(i)]
			\item there exists an open neighbourhood of $\Gamma_0$ whose image under $h$ is not contained in $H$; or 
			\item at the minimum distance where the tropical curve leaves the plane $H$, it leaves it along 3 edges. 
		\end{enumerate}
	\end{definition}

	\begin{example}
		\begin{enumerate}
			\item A non-superabundant tropical curve is well-spaced. 
			\item (Cycle spans a plane). See Figure \ref{fig:well-spacedness}. 
			\item (Cycle spans a line). Suppose we have a configuration as in Figure \ref{fig:well-spaced2}. Here $H_1, H_2$ are the planes spanned by the edges adjacent to each vertex, and $H_1 \neq H_2$. This tropical curve well-spaced, since for every plane containing the cycle, no open neighbourhood of the cycle is contained in that plane. 
			
			\begin{figure}[h]
				\centering
				\begin{tikzpicture}
					\fill (0,0) circle[radius=2pt];
					\fill (2,0) circle[radius=2pt];
					\draw (0, 0.05) -- (2, 0.05);
					\draw (0, -0.05) -- (2, -0.05);
					\draw (0, 0) -- (-1, 1); 
					\draw (0, 0) -- (-1, -1); 
					\draw (2, 0) -- (3.3, 0.8); 
					\draw (2, 0) -- (3.3, -0.8);
					
					\draw[forestgreen] (-0.7,-0.6) --(0.7, -0.6) --(0.7, 0.6) -- (-0.7, 0.6) -- cycle;
					
					\draw[coralpink] (1.2, -0.5) -- (1.5,0.5) --(3.1,0.5)--(2.8,-0.5) --cycle;
					
					\node[forestgreen] at (-1, 0.5) {$H_1$};
					\node[coralpink] at (3.3, -0.2) {$H_2$};
					
				\end{tikzpicture}
				\caption{A well-spaced tropical curve with $\Gamma_0$ spanning a line.}
				\label{fig:well-spaced2}
			\end{figure}
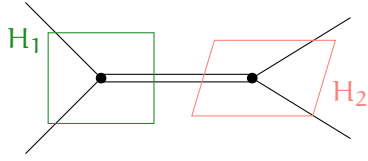
			
			\item (Cycle spans a line). Suppose we have a configuration as in Figure \ref{fig:well-spaced3}. For every plane $H$ containing the cycle $\Gamma_0$ that is not equal to $H_1$, no open neighborhood of the cycle is contained in $H$. So the only condition from well-spacedness comes from considering the plane $H_1$. This will by a condition on the edge length $L$. 
			 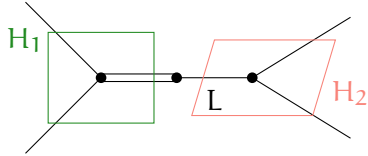
\begin{figure}[h]
				\centering
				\begin{tikzpicture}
					\fill (0,0) circle[radius=2pt];
					\fill (2,0) circle[radius=2pt];
					\fill (1,0) circle[radius=2pt];
					\draw (0, 0.05) -- (1, 0.05);
					\draw (0, -0.05) -- (1, -0.05);
					\draw (1, 0) -- (2,0);
					\draw (0, 0) -- (-1, 1); 
					\draw (0, 0) -- (-1, -1); 
					\draw (2, 0) -- (3.3, 0.8); 
					\draw (2, 0) -- (3.3, -0.8);
					
					\draw[forestgreen] (-0.7,-0.6) --(0.7, -0.6) --(0.7, 0.6) -- (-0.7, 0.6) -- cycle;
					
					\draw[coralpink] (1.2, -0.5) -- (1.5,0.5) --(3.1,0.5)--(2.8,-0.5) --cycle;
					
					\node[forestgreen] at (-1, 0.5) {$H_1$};
					\node[coralpink] at (3.3, -0.2) {$H_2$};
					\node at (1.5, -0.3) {$L$};
				\end{tikzpicture}
				\caption{A tropical curve with a single well-spaced condition.}
				\label{fig:well-spaced3}
			\end{figure}
		\end{enumerate}
		
	\end{example}
	
	\begin{remark} \label{rem:empty}
		Even if a tropical curve is tropically well-spaced, its corresponding logarithmic moduli space may be empty. For instance, Figure \ref{fig:well-spacedempty} shows a tropical curve of genus $1$ mapping to a line. It corresponds to a genus $1$ curve mapping with degree $1$ to a genus $0$ curve, which cannot happen. 
				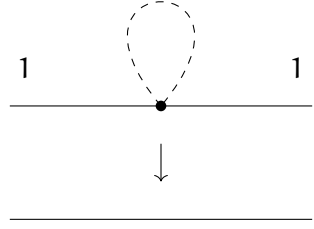
\begin{figure}[h]
			\centering
			\begin{tikzpicture}
				\fill (0,0) circle[radius=2pt];
				\draw (0,0)--(2,0);
				\draw (-2,0)--(0,0);
				\draw[dashed] (0,0)  to[in=50,out=130,loop, style={min distance=24mm}] (0,0); 
				\draw[->] (0, -0.5) -- (0,-1);
				\draw (-2,-1.5) -- (2,-1.5);
				\node at (1.8, 0.5) {$1$};
				\node at (-1.8, 0.5) {$1$};
			\end{tikzpicture}
			\caption{Tropical curve mapping to $\mathbb{R}$ which is tropically well-spaced but not realizable.}
			\label{fig:well-spacedempty}
		\end{figure}

	\end{remark}
	
	The locus of tropical curves satisfying well-spacedness cuts out cones in $\mathcal{M}^{\trop}_{\Lambda}$ of the expected dimension.

	\subsection{Logarithmic well-spacedness}

	We review the construction of the moduli space of radially aligned well-spaced logarithmic curves constructed in \cite{RSW19II}. This is a logarithmically smooth compactification of $\mathcal{\overline{M}}_{1,n}(X)^\circ$, the interior of the main component of $\mathcal{\overline{M}}_{1,n}(X)$.

	Let $C \rightarrow S$ be a genus 1 logarithmic curve with tropicalization $\Gamma \rightarrow \mathbb{R}^3$. Let the logarithmic structure on $S$ be denoted by $M_S$.\footnote{For a picture, one should take $S = \mathbb{A}^1$ with usual logarithmic structure, and the edge lengths to be valued in $\mathbb{R}_{\geq 0}$.}  For each vertex $v \in V(\Gamma)$ the distance to the circuit $\Gamma_0$ gives a well-defined piecewise linear function, which we denote by $\lambda: V(\Gamma) \rightarrow \overline{M}_S$. 
	There is a partial ordering on $\overline{M}_S$, where $a \geq b$ if $a-b \in \overline{M}_S \subset \overline{M}_S^{gp}$.

	\begin{definition}
		A family of genus 1 logarithmic curves is \emph{radially aligned} if $\lambda(v)$ and $\lambda(w)$ are comparable for all geometric points of $s \in S$ and $v, w \in \Gamma_s$. 
	\end{definition}

	The moduli space of radially aligned logarithmic curves is a logarithmic modification of the moduli space of logarithmic curves, obtained by subdividing the associated cone complex along loci of tropical curves where $\lambda(v) = \lambda(w)$ for some vertices $v,w$. Well-spacedness will give a closed condition.
	
	Let $f: C \rightarrow Z$ be a map from a radially aligned logarithmic curve to a toric variety $X$ with dense torus $T$.

	\begin{definition}[Logarithmic well-spacedness]
		Let $X$ be a toric variety. A radial logarithmic map $f: C \rightarrow Z$ is \emph{well-spaced} if $f$ satisfies the \emph{factorization property} for all subtori $H$ of $T$. 
	\end{definition}

	The definition of the factorization property is given in \cite[Definition 3.4.1]{RSW19II}. We will not state it precisely here, and instead lay out some exposition. 
	
	Consider a genus $1$ logarithmic curve mapping to a toric variety $X$. Suppose that the tropicalization $f: \Gamma \rightarrow \mathbb{R}^r$ contracts the circuit $\Gamma_0$. Let $\delta_f$ be the function given as the minimum distance from from $\Gamma_0$ to where $\Gamma$ is not contracted. Contracting everything inside a radius of $\delta_f$ in $\Gamma$ (adding vertices as necessary) gives a new tropical curve with a genus $1$ vertex. This induces a logarithmic modification 
	$$\tau: \tilde{C} \rightarrow C$$
	and a contraction 
	$$\gamma: \tilde{C} \rightarrow \overline{C}$$
	to a curve with a Gorenstein elliptic singularity. We say that the map $f: C \rightarrow X$ has the \emph{factorization property} if the induced map $f: \tilde{C} \rightarrow X$ factors through $\overline{C}$. 
	
	\begin{remark}
		The factorization condition can be more explicitly stated as the following. Suppose that $\tilde{C}$ has components $\tilde{C}_i$ which are glued at points $p_i$, and that $\tilde{C}$ has an elliptic singularity at this glued point. The factorization condition is equivalent to requiring that the tangent vectors to $\tilde{C}_i$ at $p_i$ are linearly dependent. We refer to \cite[Section 3.2]{KS24} for details. 
	\end{remark}
	
	For each codimension 1 subtorus $H$ of $T$, there is an expansion $\tilde{X}$ of $X$ and a morphism from $\tilde{X}$ to a chain of $\mathbb{P}^1$s extending the quotient of $T$ by $H$, such that if the circuit of $C$ is contracted, it is contracted to the interior of one of the $\mathbb{P}^1$s. The \emph{well-spacedness} condition asks that for all codimension $1$ subtori, the induced morphism satisfies the factorization condition. 

	\begin{theorem} 
		\cite[Theorem B]{RSW19II}.
		Consider the following data as a moduli problem on logarithmic schemes:
		\begin{enumerate}
			\item a family of $n$-marked, radially aligned logarithmic curves $C \rightarrow S$; and 
			\item a logarithmic stable map $f: C \rightarrow X$ with contact order $\Lambda$,
			such that the map $f$ is well-spaced. 
		\end{enumerate}
		The moduli problem is represented by a proper and logarithmic smooth stack with logarithmic structure, denoted $\mathcal{W}_\Lambda(X)$. 
	\end{theorem}
	
	In \cite[Section 4.6]{RSW19II}, it is shown that logarithmic well-spacedness condition is equivalent to the tropical well-spacedness condition.

	\begin{remark}
		While this space was defined in the collapsed logarithmic stable maps setting, it naturally caries over to the expanded setup. The resulting moduli space, which by abuse of notation we continue to call $\mathcal{W}_\Lambda$, is a logarithmic modification of the original. These spaces define the same invariants. See \cite[Remark 3.3.2]{RSW19II}.
	\end{remark}

	\subsection{Well-spaced invariants} \label{sec:winv}
	
	We define the well-spaced invariants. As before, fix integers $a,b$ such that 
	$$a+2b= \int_\beta c_1(T_X).$$ 
	Let $\varphi_1, \ldots, \varphi_a \in A^{n-1}(X)$ and $\varphi_{a+1}, \ldots, \varphi_{a+b} \in A^n(X)$ be cohomology classes on $X$.  There is an evaluation morphism 
	$$\ev_W: \mathcal{W}_\Lambda \rightarrow (\mathbb{P}^3)^n$$
	which can be obtained by blowing down the expanded targets.

	\begin{definition}
		For $a,b = (a,b)$ satisfying $a +2b = 4d$, the \emph{well-spaced invariants} are
		$$\Winv = \int_{[\mathcal{W}_{\Lambda}(X)]^{\vir}} \ev_{m+1}^*(\varphi_1) \cdots  \ev_{m+n}^*(\varphi_{a+b})$$
		with $\varphi_1, \ldots, \varphi_a$ classes in $A^{r - 1}(X)$ corresponding to tropical cycles $\ell_1^{\trop}, \ldots, \ell_a^{\trop}$, and point classes $\varphi_{a+1}, \ldots, \varphi_{a+b}$.
	\end{definition}

	\begin{remark} \label{rem: P3 enumerative}
		There is subtlety in what we mean by an ``enumerative'' count. Such a count is not always well-defined, since the pull back of classes may not intersect transversely. For $X = \mathbb{P}^3$, the number $n_{a,b}$ of genus $1$ curves through $a$ lines and $b$ points is well-defined \cite{Vakil97}. For generic point and line conditions, any curve that meets the conditions will have transverse intersection with the boundary. There are $(d!)^4$ ways to label the intersection points, and thus the corresponding well-spaced invariant satisfies $W_{a,b} = (d!)^4 \times n_{a,b}$.  
	\end{remark}

	Fix a well-spaced tropical type $\gamma$. In joint work with Cela \cite{CK26}, we define \emph{weights} $w_\gamma$ and prove the following tropical correspondence theorem for well-spaced curves.

	\begin{definition}
		Fix $a$ tropical line and $b$ point conditions in $\mathbb{R}^3$. The tropical well-spaced invariant $W_{a,b}$ is the weighted sum 
		$$\Winv^{\trop} = \sum_{\gamma} M_{\gamma} \cdot w_{\gamma}$$
		where the sum goes over the tropical types $\gamma$ of the well-spaced tropical curves meeting the point and line conditions.  
	\end{definition}

	\begin{theorem}\label{thmB} \cite[Theorem B]{CK26}
		The number $\Winv^{\trop}$ is well-defined and 
		$$\Winv= \Winv^{\trop}.$$
	\end{theorem}

	While explicit combinatorial multiplicities are not necessary for Theorem \ref{thmA}, we note down the relevant cases to use in Section \ref{sec:examples}. As explained in \cite{CK26}, for explicit computations it is best to work with tropical curves in $\mathbb{R}^3$, forgetting the fan structure coming from $X$. These give \emph{different} weights (not $w_{\gamma}$). 
	
	\begin{definition}
		Let $e_1, \ldots, e_t$ be edges forming the cycle in $\gamma$, with directions $u_i$. Consider the map 
		$$\mathbb{Z}^t \rightarrow \mathbb{Z}^3 \cap \operatorname{Span}_{\mathbb{R}}(h(\gamma_0))$$
		defined by $(\ell_1, \ldots, \ell_t) \mapsto \sum \ell_i u_i$. The \emph{loop multiplicity} is the index of this map, i.e. the index of the image of $\mathbb{Z}^t$ inside  $\mathbb{Z}^3 \cap \operatorname{Span}_{\mathbb{R}}(h(\gamma_0))$. 
	\end{definition}
	
	For well-spaced tropical curves in $\mathbb{R}^3$, the relevant multiplicities (specialized from \cite[Theorem D]{CK26}) are given as follows. 
	\begin{enumerate}
		\item (Excess dimension 0.) The multiplicity is the loop multiplicity. 
		\item (Excess dimension 1.) The cycle is contained in a plane $H$, but not in a line. 
		\begin{enumerate}
			\item Suppose the neighborhood of the cycle is also contained in $H$. Then the weight of $\gamma$ is the loop multiplicity.
			\item Suppose $\gamma$ has exactly one four valent vertex $v_1$ in the cycle, whose neighborhood is not contained in a plane $H$. Let $\pm u$ be the direction of the edges adjacent to $v_1$ not in the cycle (directing the edges away from $v_1$). With a change of coordinates, we can assume that the plane $H$ is parallel to the $x$-$y$ plane. Let the absolute value of the $z$-coordinate of $u_e$ be $a$. The weight of $\gamma$ is 
			$$ \text{loop multiplicity} \times ( a - 1).$$
		\end{enumerate}
		
		\item (Excess dimension 2.) The cycle is contained in a line $L$, and must consist of two vertices $v_1, v_2$ joined by two lines, possibly with bivalent vertices and marked points in the middle.
		\begin{enumerate}
			\item Suppose $v_1, v_2$ are both four valent. Let $\pm u_1, \pm u_2$ be directions of the adjacent edges to $v_1$, resp. $v_2$ \emph{not} in the cycle. Let $\bar{u}_i$ be the image of $u_i$ under $\mathbb{Z}^3 \rightarrow \mathbb{Z}^3/(\mathbb{Z}^3 \cap L)$. Let $i(P)$ be the  Then the number of interior lattice points in the parallelogram spanned by $\bar{u}_1, \bar{u}_2$. The multiplicity is 
			$$\frac{1}{\Aut(\gamma)} \times \text{loop multiplicity} \times  i(P).$$
			\item Suppose $v_1$ is four valent and $v_2$ is trivalent. Let $\pm u$ be the direction of the edges adjacent to $v_1$ not in the cycle. Let $\bar{u}$ be the image of $u$ under the quotient $\mathbb{Z}^3 \rightarrow \mathbb{Z}^3/(\mathbb{Z}^3 \cap L)$ and let $a$ be the weight of $\bar{u}$. The multiplicity is 
			$$\frac{1}{\Aut(\gamma)} \times \text{loop multiplicity} \times  (a-1).$$
		\end{enumerate}
		
		\item (Excess dimension 3.) The multiplicity is $1/\Aut(\gamma)$. 

	\end{enumerate}
	
	By the results of section \ref{sec:calculations}, this covers all cases.

	\section{Tropical contributions and virtual multiplicities} \label{sec:calculations}
	
	In this section, we find all contributing tropical curves for sufficiently general point conditions, and prove Theorem \ref{thmA}.  The structure of this section is as follows: we start with some setup in Section \ref{sec:generalconditions}. In Section \ref{sec:nonsupcont}, we show that the contributions of non-superabundant curves of the well-spaced and logarithmic Gromov-Witten count coincide. In Section \ref{sec:excess1}, we show that for tropical types $\gamma$ of excess dimension 1, the logarithmic Gromov--Witten count vanishes (while well-spaced contributions may exist). In Section \ref{sec:exdim2}, we study tropical types $\gamma$ of excess dimension 2, and show that again there are well-spaced curves which are not rigid as parametrized tropical curves and contribute to the well-spaced count, while the logarithmic Gromov--Witten contribution vanishes. In Section \ref{sec:exdim3}, we show that the well-spaced contributions vanish and the logarithmic Gromov--Witten contributions sum to $E_0$, as in Theorem \ref{thmA}. This concludes the proof.

	\subsection{Generality conditions} \label{sec:generalconditions}
	
	Recall the definition of an affine condition in Section \ref{sec:decomposition}. There are several generality conditions we wish to impose. We make this precise. 
	
	\begin{definition}
		A \emph{general condition} is a dense open set of the space of affine conditions $(\ell_i, p_j)$. 
	\end{definition}
	
	We can intersect finitely many general conditions to obtain a general condition.

	\begin{definition}
		Let $h: \Gamma \rightarrow \RR^3$ be a tropical curve passing through affine conditions $(\ell_i, p_j)$. Say this configuration $(h, \ell_i, p_j)$ is \emph{general} if there exists a neighbourhood of the condition space $U \subset (\RR^3)^{a+b}$ (the first $a$ coordinates mark the vertex of the tropical line), such that for any $(\ell_i', p_j') \in U$, there is a deformation of $h$ preserving the tropical type which meets it.   
	\end{definition}
	
	\begin{definition}
		An affine condition $(\ell_i, p_j)$ is \emph{general for tropical type} if for every tropical curve $(h, \Gamma)$ which meets it, $(h, \ell_i, p_j)$ is a general configuration. 
	\end{definition}
	
	In terms of maps of cone complexes (as in Section \ref{sec:decomposition}), we are choosing points and lines such that the image of the cone mapping to them meets the conditions on its interior. The failure of this condition has positive codimension, so this forms a dense open set. Hence, being general for tropical type defines a general condition. 
	
	There are many immediate consequences for choosing an affine constraint which is general for tropical type. To demonstrate, suppose that $(h, p_i)$ is a general configuration (here, $a = 0$). Then two points cannot lie on an edge sharing a vertex. Indeed, such points would be constrained to lie in a given plane, which is not general.

	\begin{definition}
		Let $h: \Gamma \rightarrow \RR^3$ be a tropical curve passing through affine conditions $(\ell_i, p_j)$. An edge $E$ of $\Gamma$ is \emph{fixed by the conditions} if for every deformation of $h$ meeting the conditions and fixing the tropical type, the length $\ell(E)$ of $h(E)$ is fixed. 
	\end{definition}

	\begin{definition}
		An affine condition $(\ell_i, p_j)$ is \emph{general for edge lengths} if for every tropical curve $(h, \Gamma)$ which meets it, the following condition holds: for any two distinct collections of fixed edges $\mathcal{E}$ and $\mathcal{F}$, have
		$$\sum_{E \in \cE} \ell(E) \cdot u_E \neq \sum_{F \in \cF} \ell(F) \cdot u_F.$$
		
	\end{definition}
	
	\begin{example} 
		The tropical curve shown in Figure \ref{fig:general-edge-length} has two fixed edges, as indicated. They are different lengths and the configuration is general for edge lengths. The vertex at the centre may be moved along the edge. Although this may cause the paths to the two marked points to have the same length, we do not count this, since such paths contain edges which are not fixed. 
		
		\begin{figure}[h]
			\centering
			\begin{tikzpicture}
				\draw (0, 0) -- (2, 0);
				\draw (0, 0) -- (-1, 1); 
				\draw[thick, forestgreen] (0, 0) -- (-0.3, -0.3); 
				\draw (-0.3, -0.3) -- (-1, -1); 
				\draw[thick, forestgreen] (2, 0) -- (2.7, 0.7); 
				\draw (2.7, 0.7) -- (3,1);
				\draw (2, 0) -- (3, -1);
				\draw [coralpink] plot [only marks, mark=square*] coordinates {(2.7,0.7) (-0.3, -0.3)};
				\fill (0,0) circle[radius=2pt];
				\fill (2,0) circle[radius=2pt];
				\fill (1,0) circle[radius=2pt];
			\end{tikzpicture}
			\caption{Tropical curve with general edge lengths. Two point conditions are given by the pink squares. The fixed edges are indicated in green.}
			\label{fig:general-edge-length}
		\end{figure}
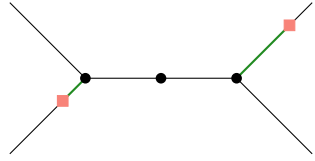

	\end{example}

	\begin{lemma}
		Generality for edge lengths is a general condition. 
	\end{lemma}
	
	\begin{proof}
		The equalities above define codimension one subspaces of the relevant cone in the moduli space of tropical maps. 
	\end{proof}
	
	We will now assume that our affine conditions are general for tropical type and edge lengths.

	\subsection{Non-superabundant contributions} \label{sec:nonsupcont}
		
	\begin{proposition} \label{prop:nonsup}
		Suppose $\gamma$ is a non-superabundant tropical type. Then the logarithmic Gromov--Witten and well-spaced contributions are equal. 
	\end{proposition}
	
	\begin{proof}
		 To obtain the well-spaced count, we perform a logarithmic modification and impose the well-spacedness constraint, and evaluate the (virtual) fundamental class of the stratum. If a stata corresponds to a non-superabundant tropical curve, then there is no modification or conditions, and the invariants given are the same.
	\end{proof}

	\subsection{Excess dimension 1} \label{sec:excess1}
	
	Superabundant curves cannot be avoided via the generality of the point and line conditions, as demonstrated in Figure \ref{fig:supabgenpos}. Here there are 14 points. A consequence of Corollary \ref{cor:logandordinary} is that for any 14 points, there are superabundant curves meeting them.

	\begin{figure}[h]
		\center
		\includegraphics[width=0.7\textwidth]{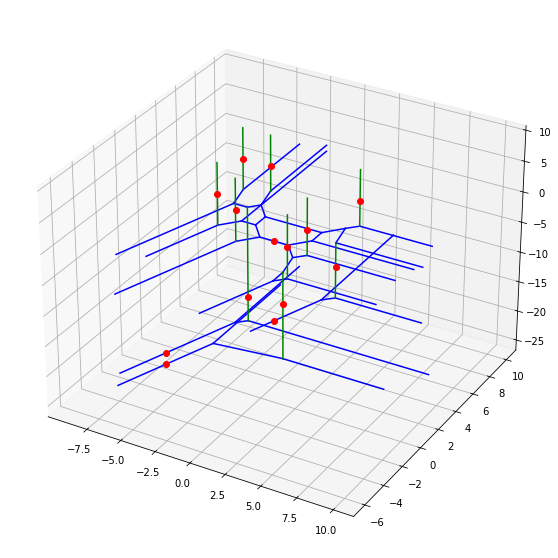}
		\caption{A general configuration with a floor decomposed superabundant tropical curve.}
		\label{fig:supabgenpos}
	\end{figure}
	
	When there are superabundant curves meeting the point conditions, we obtain a higher dimensional family of tropical curves which meet the conditions. These families can be complicated, containing several tropical types.

	We start with a vanishing statement. 
	
	\begin{lemma} \label{lem:vanish}
		Suppose that $\gamma$ is a tropical type of genus one which is rigid for the affine conditions such that a neighbourhood of the cycle is contained in a plane $H$. Suppose also that the strata corresponding to $\gamma$ has excess dimension $1$. 
		Then the virtual multiplicity, and hence the logarithmic Gromov--Witten contribution vanishes.
	\end{lemma}

	\begin{proof}
		Without loss of generality, let the plane $H$ be the $x$-$y$ plane. We have the following gluing diagram:
		\begin{center}
	\begin{tikzcd}
		\mathcal{G}_\gamma \arrow[d] \arrow[r]               & \bigtimes_{v \in V(G)} \mathcal{M}_v \arrow[d, "ev"] \\
		\bigtimes_e \cD_e \arrow[r, "\triangle"] & \big(\bigtimes_e \cD_e \big)^2                        
	\end{tikzcd}
	\end{center}
		where $G$ is the underlying graph of $\gamma$. 
		
		On the interior of $\mathcal{G}_\gamma$, the excess dimension is given by a degenerate gluing condition around the loop. Since an open neighbourhood of the cycle of the tropical curve is contained in a plane, it follows that the corresponding components of the logarithmic curve has no degree in the third $\mathbb{C}^*$ direction. Hence these components are contained at a fixed height (which may be wholly in expansions). The heights are matched along each edge in the loop, but the last matching is trivial, and this gives one excess dimension. By assumption, this is the only excess dimension. The excess bundle is the pullback of the tangent space to $\mathbb{C}^*$, coming from this direction, and hence is trivial. 
		
		The result is obtained by extending this argument to cycles. We want to show that 
		$$\triangle^![\bigtimes_{v \in V(G)} \mathcal{M}_v]$$ vanishes. 
		
		We may use Minkowski weights to calculate intersections in $\big(\bigtimes_e \cD_e \big)^2$. For a vertex $v$ in the cycle, suppose $e_1, e_2$ are the adjacent edges that are also contained in the cycle. The tropicalization of $\mathcal{M}_{v} \rightarrow D_{e_1} \times D_{e_2}$ has image contained in the set
		$$A = \{((x_1,y_1, z_1), (x_2,y_2, z_2)) | z_1 = z_2\} \subset (\mathbb{R}^3)^2.$$
		Thus the tropical image of $\bigtimes_{v \in V(G)} \mathcal{M}_v \rightarrow (\bigtimes_e \mathcal{D}_e)^2$ can be generically displaced so that it does not meet the tropical image of $\triangle$. Thus the intersection in $\big(\bigtimes_e \cD_e \big)^2$ vanishes. The result follows from applying Corollary \ref{cor:vanishing}. 
	\end{proof}

	The dimension of the tropical moduli space is at least the dimension of the strata and so the dimension condition holds if the tropical excess dimension is $1$.

	Finally, consider the case where the neighbourhood of the cycle spans (i.e. there is a four valent vertex on the cycle whose neighbourhood spans $\mathbb{R}^3$). In this case, the curve is well-spaced. The contributions are the same:

	\begin{lemma} \label{lem:neighbourhoodspans}
		Suppose that $(h, \Gamma)$ is a parametrized tropical curve of combinatorial type $\gamma$ meeting $a$ line and $b$ fixed point conditions, where the neighbourhood of the cycle $\Gamma_0$ spans $\mathbb{R}^3$. Then the logarithmic and well-spaced contributions are equal. 
	\end{lemma}

	\begin{proof}
		By dimension considerations, the tropical type is rigid (in the tropicalization of $\Mspace$). The cycle does not lie in any plane. Similarly to Proposition \ref{prop:nonsup}, the moduli spaces are the same.
	\end{proof}

	\subsection{Excess dimension 2} \label{sec:exdim2}

	\begin{proposition} \label{prop:ex2} 
		Let $(h, \Gamma)$ be a parametrised tropical curve of excess dimension $2$ meeting general point and line conditions (so $h(\Gamma_0)$ is contained in a line). Then $\Gamma$ contains one of the graphs in Figure \ref{fig:twochoices} as a subgraph, and all other vertices are trivalent. Moreover, 
		\begin{enumerate}[(i)]
			\item if $(h,\Gamma)$ is well-spaced, it has the left graph as a subgraph; and
			\item if $(h,\Gamma)$ is rigid, it has the right graph as a subgraph. 
		\end{enumerate}
	\end{proposition}

		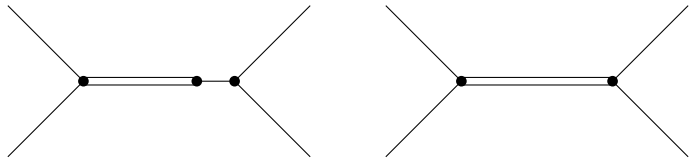
\begin{figure}[h]
		\centering
		\begin{tikzpicture}
			\fill (0,0) circle[radius=2pt];
			\fill (2,0) circle[radius=2pt];
			\fill (1.5,0) circle[radius=2pt];
			\draw (0, 0.05) -- (1.5, 0.05);
			\draw (0, -0.05) -- (1.5, -0.05);
			\draw (1.5, 0) -- (2,0);
			\draw (0, 0) -- (-1, 1); 
			\draw (0, 0) -- (-1, -1); 
			\draw (2, 0) -- (3, 1); 
			\draw (2, 0) -- (3, -1);

			\fill (5,0) circle[radius=2pt];
			\fill (7,0) circle[radius=2pt];
			\draw (5, 0.05) -- (7, 0.05);
			\draw (5, -0.05) -- (7, -0.05);
			\draw (5, 0) -- (4, 1); 
			\draw (5, 0) -- (4, -1); 
			\draw (7, 0) -- (8, 1); 
			\draw (7, 0) -- (8, -1);

		\end{tikzpicture}
		\caption{Two cases of well-spaced tropical types meeting conditions with excess dimension $2$.}
		\label{fig:twochoices}
	\end{figure}	
	
	We start with some lemmas. 
	
	\begin{lemma}\label{lem:no4}
		Let $h: \Gamma \rightarrow \mathbb{R}^r$ be a genus 0 parametrized tropical curve meeting a general affine condition. Then 
		\begin{enumerate}[(i)]
			\item the valence of each vertex is at most $3$; and
			\item no two infinite ends lie on the same line. 
		\end{enumerate} 
	\end{lemma}
	
	\begin{proof}
		Both are basic dimensions counts. See \cite[Proposition 9.4.6]{CMR23}. 
	\end{proof}
	
	\begin{lemma} \label{lem:simplify}
		If $(h, \Gamma)$ is a parametrized tropical curve meeting general conditions, and $\operatorname{Ex}(h) = 2$, then ignoring two valent vertices, the cycle consists of two vertices with two edges between them. 
	\end{lemma}
	
	\begin{proof}
		We will produce a genus 0 curve $(h', \Gamma')$ which is 4-valent and goes through these general points, which gives a contradiction to Lemma \ref{lem:no4}. This is done by the following procedure: along $h(\Gamma_0)$, mark every point $q_i$ which a vertex is mapped to. There are finitely many such points. Subdivide $\Gamma$ by requiring that every point in $h^{-1}(q_i)$ is a vertex. Then for each line segment $t_j$ between adjacent $q_i$, $h^{-1}(l_j)$ is a disjoint collection of edges. Glue all these edges together and take the new weight to be the sum of the weights. It is clear the resulting curve is genus 0 and the balancing condition still holds on the induced map.

		Our assumption that $\operatorname{Ex}(h)= 2$ means that the image of the cycle lives in a line $L$. Since all of the edges of the cycle are bounded, the image of the cycle $\Gamma_0$ is a closed interval. There are at least vertex mapped to each of the two boundary points, say $v_1$ and $v_2$. Suppose for contradiction that there is another vertex, say $v$, in the cycle, with valency at least three. 
		
		Each of the vertices $v_1, v_2$, and $v$ have an adjacent edge which doesn't lie on the cycle. Say their directions are $u_1$, $u_2$ and $u$. Cutting the tropical curve at these edges produces three trees. 
		
		The directions $u_1$ and $u_2$ have differing signs, so $u$ is in the positive span of one of them. Without loss of generality, this is $u_1$. By Lemma \ref{lem:no4} (ii) applied to $\Gamma'$, we cannot have both the connecting trees lying wholly in $L$, since this would produce two infinite ends lying in $L$. Let the closest point to the cycle where the tropical curve leaves $L$ be $a$ (see Figure \ref{fig:flat}). 
		
		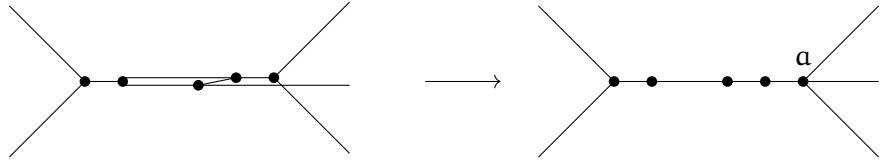
\begin{figure}[h]
			\centering
			\begin{tikzpicture}
				\fill (-0.5,0) circle[radius=2pt];
				\fill (0,0) circle[radius=2pt];
				\fill (1, -0.05) circle[radius=2pt];
				\fill (1.5,0.05) circle[radius=2pt];
				\fill (2,0.05) circle[radius=2pt];

				\draw (0, 0.05) -- (1.5, 0.05);
				\draw (0, -0.05) -- (1, -0.05);
				\draw (1, -0.05) -- (1.5, 0.05);
				\draw (1.5, 0.05) -- (2,0.05);
				\draw (-0.5, 0) -- (0, 0);
				\draw (-0.5, 0) -- (-1.5, 1); 
				\draw (-0.5, 0) -- (-1.5, -1); 
				\draw (2, 0.05) -- (3, 1.05); 
				\draw (2, 0.05) -- (3, -0.95);
				\draw (1, -0.05) -- (3, -0.05);

				\draw[->] (4, 0) -- (5,0);
				
				\fill (6.5,0) circle[radius=2pt];
				\fill (7,0) circle[radius=2pt];
				\fill (8, 0) circle[radius=2pt];
				\fill (8.5,0) circle[radius=2pt];
				\fill (9,0) circle[radius=2pt];

				\draw (6.5, 0) -- (10,0);
				\draw (6.5, 0) -- (5.5, 1); 
				\draw (6.5, 0) -- (5.5, -1); 
				\draw (9, 0) -- (10, -1);
				\draw (9, 0) -- (10, 1);

				\node at (9, 0.3) {$a$};
				
			\end{tikzpicture}
			\caption{Flattening the cycle.}
			\label{fig:flat}
		\end{figure}

		At the vertex $a$, the corresponding genus $0$ curve $\Gamma'$ must have a four valent vertex. Then Lemma \ref{lem:no4} (i) gives a contradiction.

	\end{proof}
	
	\begin{corollary} 
		If $\operatorname{Ex}(h) = 2$ and no edge is contracted by $h$, then (ignoring two valent vertices) the only rigid tropical types occur when there are two edges joining two vertices, (see right of Figure \ref{fig:twochoices}). 
	\end{corollary}

	\begin{proof}
		From Lemma \ref{lem:simplify}, the neighbourhood of the cycle looks like Figure \ref{fig:neardoubleedge}, with the lengths $\ell_1$ or $\ell_2$ potentially being zero.

		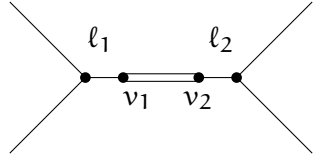
\begin{figure}[h]
			\centering
			\begin{tikzpicture}
				\fill (0,0) circle[radius=2pt];
				\fill (2,0) circle[radius=2pt];
				\fill (1.5,0) circle[radius=2pt];
				\fill (0.5, 0) circle[radius=2pt];
				\draw (0.5, 0.05) -- (1.5, 0.05);
				\draw (0.5, -0.05) -- (1.5, -0.05);
				\draw (0, 0) -- (0.5, 0);
				\draw (1.5, 0) -- (2,0);
				\draw (0, 0) -- (-1, 1); 
				\draw (0, 0) -- (-1, -1); 
				\draw (2, 0) -- (3, 1); 
				\draw (2, 0) -- (3, -1);
				
				\node at (0.2, 0.5) {$\ell_1$};
				\node at (1.8, 0.5) {$\ell_2$};
				\node at (0.7, -0.3) {$v_1$};
				\node at (1.5, -0.3) {$v_2$};

			\end{tikzpicture}
			\caption{The boundary of $\cM_\gamma$ has the tropical type shown, with $\ell_1>0$ or $\ell_2 >0$.}
			\label{fig:neardoubleedge}
		\end{figure}	
		
		Two dimensions of deformation are given by moving $v_1$ and $v_2$ along the edge, and by dimension assumption, these are all the deformations meeting the conditions. The only rigid types occur when $\ell_1 = \ell_2 = 0$.

	\end{proof}
	
	\begin{remark}
		By the same argument, the four outer edges leave $L$. This argument also shows that all edges outside of the neighbourhood of the circuit are fixed, and the rest of the vertices are trivalent. 
	\end{remark}

	\begin{lemma}\label{lem:notwellspaced}
		Let $(h, \Gamma)$ be a parametrized tropical curve with cycle $\Gamma_0$ given by two vertices connected by two edges as in Figure \ref{fig:neardoubleedge}. Then $(h, \Gamma)$ is well-spaced if and only if one of the $\ell_i$ are $0$. 
	\end{lemma}
	
	\begin{proof}
		
		Let $H_1$ be the plane spanned by the three left most vertices, and $H_2$ the plane spanned by the three right most vertices. Suppose, for contradition, that the curve is well-spaced and $\ell_1, \ell_2 > 0$. Then take $H'$ to be a plane which contains the cycle but is different from $H_1$ and $H_2$. Applying well-spacedness, we conclude that $\ell_1 = \ell_2$. 
		
		The parametrized tropical curve eventually leaves $H_1$, so there is a minimal distance from each such vertex to where the curve leaves $H_1$. Moreover, by generality of edge lengths, we may assume that these distances are distinct. Combined with the condition that $\ell_1 = \ell_2$, the minimal distance until the curve leaves the plane is met only once, and leaves at a trivalent vertex. Hence the well-spacedness condition is not met. 
	\end{proof}
	
	We finish the proof of Proposition \ref{prop:ex2} by remarking that the corresponding genus $0$ curve $(h', \Gamma')$ is rigid. By generality of edge lengths, the $(h, \Gamma)$ cannot be well-spaced if $\ell_1 = \ell_2 = 0$. 
	
	When the curve is rigid, we have two cases:
	
	\begin{enumerate}
		\item The neighbourhood of the cycle is contained in a plane $H$. In this case, generality of point conditions implies that the curve is not well-spaced. By explicit computation, as in Example \ref{exa:expansions}, the excess dimension of the strata is $1$. The logarithmic Gromov--Witten contribution vanishes by an application of Lemma \ref{lem:vanish}.
		\item The neighbourhood of the cycle is not contained in a plane $H$. Then the curve is well-spaced, and the logarithmic Gromov--Witten and well-spaced contributions coincide, as in Lemma \ref{lem:neighbourhoodspans}.
	\end{enumerate}

	\subsection{Excess dimension 3} \label{sec:exdim3}

	For a curve with excess dimension 3 passing through general conditions to be rigid there must be a genus $1$ vertex, and no such curve is well-spaced. For rigidity, this is clear. For well-spacedness, note that the image of the curve is a genus $0$ curve through the same conditions, which for dimension reasons must be trivalent, and by the generality of the conditions has no two edge-lengths the same. Thus if there was a contracted cycle, the neighborhood of the cycle could not span, and the well-spacedness condition could not be satisfied. 
	
	We calculate the contribution to the logarithmic Gromov--Witten invariant. Suppose there is a genus 1 vertex. Then the underlying graph is a tree. By Lemma \ref{lem:no4}, this tree is trivalent. The genus 1 vertex is either bivalent or trivalent. The two cases are shown in Figure \ref{fig:g1vertex}.

		\begin{figure}[h]
			\centering
			\begin{tikzpicture}
				\fill (0,0) circle[radius=2pt];
				\fill (2,0) circle[radius=2pt];
				\fill (1,0) circle[radius=2pt];
				\node at (1.1,0.3) {$1$};
				\draw (0, 0) -- (2, 0);
				\draw (0, 0) -- (-1, 1); 
				\draw (0, 0) -- (-1, -1); 
				\draw (2, 0) -- (3, 1); 
				\draw (2, 0) -- (3, -1);

				\fill (6,0) circle[radius=2pt]; 
				\draw (4.8,0) -- (6, 0); 
				\draw (6,0) -- (7, 1); 
				\draw (6,0) -- (7, -1); 
				\node at (6, 0.3) {$1$};
				
			\end{tikzpicture}
			\caption{Genus 1 vertex at 2-valent or 3-valent vertex.}
			\label{fig:g1vertex}
		\end{figure}
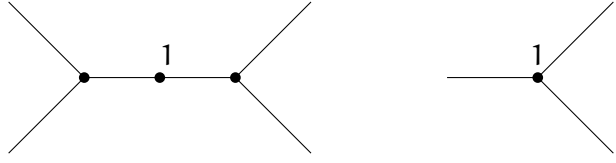

		We first want to reduce to the single genus one vertex. Let the genus one vertex be $v_1$ and label the other vertices $v_i$ for $i = 2, \ldots, |V(\gamma)|$. The source curve is given by a genus one curve $C_1$ and genus zero curves $C_i$ for $i = 2,  \ldots, |V(\gamma)|$. By generality of the conditions, each vertex is either trivalent, or bivalent and contained in a codimension $1$ cone of $\Sigma_X$.  The component $C_i$ corresponding to any genus 0 vertex $v_i$ is either
		\begin{enumerate}
			\item isomorphic to $\mathbb{P}^1$ with three points and $C_i$ is mapped to a torus fixed point after collapsing expansions; or  
			\item is isomorphic to $\mathbb{P}^1$ with two points and $C_i$ is a cover of a torus fixed line after collapsing expansions, with full ramification at the two special points. 
		\end{enumerate} 
	 	
	 	In particular, the $C_i$ for $i \geq 2$ are unique up to isomorphism. Once the curve $C_1 \in \overline{\mathcal{M}}_{1,3}$ has been chosen, the locus in  $\bigtimes_{v \in V(\gamma)} \mathcal{M}_v $ satisfying the gluing condition are parametrized by the rubber torus. We use the point and line conditions to cut this down.

	 	 A description of the rubber given in \cite{CN24}, which we use to give more explicit formulas. An expansion of $X$ is given by a polyhedral subdivision of $\Sigma_X$ whose vertices $v$ correspond to the components $X_v$ of the expansion. There is a tropical moduli space $\tau$ parametrizing the edge lengths of the subdivision. The rubber torus is the torus associated to the tropical moduli space $T_\tau = \tau \otimes \mathbb{G}_m$ \cite[Theorem A]{CN24}. As in Equation \eqref{eqn:ev'}, there is a tropical evaluation map 
	 	$$\tau \rightarrow \prod_{i=1}^a (\mathbb{Z}^3/\ell_i') \otimes (\mathbb{Z}^3)^b = N_{\mathrm{Ev}'}.$$
	 	Let $N_{\sigma_v} \rightarrow N_{\Ev'}$ be the associated map on lattices. This has lattice index, say $L$. 
	 	
	 	Let $T_{\Ev} = N_{\Ev'} \otimes \mathbb{G}_m$. Then 
	 		$$T_\tau \rightarrow T_{\Ev'}$$
	 	is a degree $L$  cover. Hence, on the interior of the strata, there are $L$ points in $\bigtimes_{v \in V(\gamma)} \mathcal{M}_{v} $ whose genus one component is $C_1$, satisfying the gluing and meeting conditions. The point and line conditions force every curve component $C_i$ to map into the interior of the corresponding component $X_v$. So $\mathcal{G}_\gamma$ is finite over the locus of curves in $\overline{\mathcal{M}}_{1,3}$ admitting meromorphic functions specified by the edge directions of $v_1$.

	 	We have reduced to a double ramification cycle calculation. There is a formula for DR cycles obtained by Jande--Pandharipande--Pixton--Zvonkine \cite{JPPZ17}, implemented in the Sage package, admcycles \cite{admcycles}. It is \emph{not} true that the higher double ramification cycles are products of DR cycles. In genus 1, there is an explicit formula \cite{HMPPS25}, but since we are working with a very simple case, we obtain a simpler result. 
	
		\begin{lemma}
			For any vectors $a = (a_i)_i, b = (b_i)_i \in \mathbb{Z}^n$ with $\sum a_i = \sum b_i = 0$, we have
			$$\int_{[\overline{\mathcal{M}}_{1,3}]} \TC_1(a,b,0) = \int_{[\overline{\mathcal{M}}_{1,3}]} -\lambda_1 \DR_1(a) \DR_1(b) $$
		\end{lemma}
	
		\begin{proof}
			We first show we can factorise out $DR_1(0)$. It is shown in \cite{RM21} that while the higher DR cycles are not products of DR cycles, they are under an appropriate logarithmic modification. To construct this modification, we take the morphism of cone complexes $\mathcal{M}_\Lambda(\mathbb{R}^3) \rightarrow \mathcal{M}_{1,n}^{\trop}(\mathbb{R}^3)$ as considered in Section \ref{sec:decomposition} when forming the space of rubber logarithmic morphisms $R_{1,n, \Lambda}$. We do the necessary subdivisions on the cone complex of $\overline{\mathcal{M}}_{1,n}$ to make this an inclusion of a subcone, which gives a logarithmic modification $\widetilde{\mathcal{M}}_{1,n}$ of $\overline{\mathcal{M}}_{1,n}$. 
			
			Since the third vector in our toric contact cycle is $0$, it imposes no further subdivisions (every tropical curve admits a degree $0$ map to $\mathbb{R}^3$, namely the constant map). In $\tilde{\mathcal{M}}_{1,n}$, the strict transform of $TC_1(a,b,0)$ is given as the product of the strict transform of $TC_1(a,b)$ and the strict transform of $DR_1(0)$. However, the strict transform of $DR_1(0)$ is pulled back from $\overline{\mathcal{M}}_{1,n}$. Thus, pushing down to $\overline{\mathcal{M}}_{1,n}$, we obtain $TC_1(a,b,0) = TC_1(a,b)DR_1(0)$.
			
			Now $DR_1(0)$ is known to be $-\lambda_1$ \cite[Section 0.5.3]{JPPZ17}. Finally,
			$$\int_{[\overline{\mathcal{M}}_{1,3}]} \lambda_1 h\DR_1(a,b) = \int_{[\overline{\mathcal{M}}_{1,3}]} \lambda_1 \DR_1(a) \DR_1(b) $$
			since $\lambda_1$ vanishes away from the locus of curves of compact type. 
		\end{proof}
	
		We are ready to calculate the contributions. Suppose there is a genus 1 vertex which is bivalent. Then both the logarithmic and well-spaced contributions are 0. Indeed, in the matrix of contact orders at the bivalent vertex, with respect to appropriate basis, there are two columns of ones. This implies the relevant invariant contains a factor of $\DR_1(0)^2 = \lambda_1^2 = 0$, and this does not contribute. 
		
		Suppose there is a genus 1 vertex at a trivalent vertex. Let the corresponding genus $0$ combinatorial type (where the genus $1$ label is turned into a genus $0$ label) be $\gamma'$. 
		
		\begin{lemma}
			We have 
			$$m_{\gamma}^{\vir} = - \frac{1}{24}(a_1 b_2 - a_2 b_1) m_{\gamma'}^{\vir}$$
			where $(a_1, b_1)$ and $(a_2, b_2)$ are the direction vectors of any two edges adjacent to $v_1$. 
		\end{lemma}
		
		\begin{proof}
		 	Write $ a = (a_1, a_2 , -a_1-a_2), b = (b_1, b_2, -b_1-b_2)$.Applying admcycles\footnote{In fact, this is a calculation given as an example in \cite[Secton 4.1]{admcycles}.}, we obtain Then 
			$$\int_{[\overline{\mathcal{M}}_{1,3}]} -\lambda_1 \DR_1(a) \DR_1(b) = -\frac{1}{24}(a_1b_2-a_2b_1)^2$$
		
			The corresponding genus $0$ calculation gives $1$. The factor $L$ from above, as well as the factor $\frac{1}{\ell(\gamma)} \prod_{e \in E(\gamma)} w(e)$ in Theorem \ref{thm:split} does not depend on the genus markings. 
		
		\end{proof}
		
		Summing the contributions gives the following result:

		\begin{proposition}
			The total contribution to the logarithmic Gromov--Witten invariant coming from genus $1$ parametrised tropical curves with a genus $1$ vertex is
			$$-\frac{1}{24}E_0 = -\frac{1}{24}\sum_{\gamma}\sum_{v \in V(\gamma)} |u_1 \wedge u_2|^2M_\gamma,$$
			where the sum is over genus $0$ tropical curves going through the same affine conditions. 
		\end{proposition}

		\begin{definition}
			For each vertex $v$, call $(a_1b_2-a_2b_1)^2$ the \emph{vertex multiplicity} of $v$. 
		\end{definition}
		
		The well-spaced contribution vanishes by the following lemma. 

		\begin{lemma}
			Suppose $h: \Gamma \rightarrow \mathbb{R}^r$ is a genus 1 tropical curve with a single genus 1 vertex meeting general conditions. Then $(h, \Gamma)$ is not well-spaced. 
		\end{lemma}
		
		\begin{proof}
			Similarly to Lemma \ref{lem:notwellspaced}, this follows from the generality of edge lengths. 
		\end{proof}
		
	This completes the proof of Theorem \ref{thmA}.

	\section{Examples} \label{sec:examples}
	
	In this section, we assume $X = \mathbb{P}^3$. We give methods for calculation and example computations. 
	
	\subsection{Floor diagrams} \label{sec:floor}
	
	We use floor diagrams to obtain a method for calculating $E_0$. 
 For genus $0$ curve counts in toric surfaces, \emph{floor diagrams} are obtained by stretching the point conditions in a chosen direction \cite{BM08}. Brugall{\'e}--Mikhalkin developed a generalisation of these techniques for higher dimensional targets \cite{BM07}. We will review this in the case of $\mathbb{P}^3$.
	
	\begin{definition}
		For any parametrised tropical curve, call the vertical edges \emph{elevators}. Removing the vertical edges, call the connected components of the resulting graph \emph{floors}. 
	\end{definition}
	
	Let $X \cong \mathbb{P}^3$. Consider tropical lines in $\mathbb{R}^3$ consisting of a vertex and four infinite ends.\footnote{This is a degenerate tropical line, in the sense of Mikhalkin. There is no difference in enumeration.} For a point condition, we will say that it has a (single) \emph{vertex} at the point. We can consider point and line conditions where the vertices of the conditions are contained in some cylinder 
	$$\mathcal{H} = \{|x|< \epsilon, |y| < \epsilon\}.$$ 
	We also require that the vertical positions of the vertices of any two distinct conditions have a difference of at least $E$. We may also position the conditions such that all the point conditions are above all the line conditions\footnote{Our orientation is swapped from Brugall{\'e}--Mikhalkin.}. 
	
	\begin{definition}
		Suppose we have a configuration $(h, \ell_i, p_j)$, where $\ell_i$ are tropical lines. Then a condition $\ell_i$ is \emph{vertical} if it meets the tropical curve on its vertical edge, and \emph{horizontal} otherwise. 
	Say that any point condition is horizontal. 
	\end{definition}

	Fix a degree $d$. For \emph{genus $0$} tropical curves, it can be shown that if $\epsilon << E$, then every degree $d$ tropical curve meeting the conditions is \emph{floor decomposed} i.e. it has the property that each floor meets exactly one horizontal condition. 
	
	As a consequence, the tropical curves meeting these conditions can be found using \emph{floor diagrams}. 
	
	\begin{definition}
		Let $\mathcal{D}$ be an oriented tree with $d$ infinite ends which are oriented outwards. For each vertex of $\mathcal{D}$,  the \emph{divergence} is the number of edges going outwards minus the number of edges coming inwards. If the divergence for each vertex is positive, call $\mathcal{D}$ a \emph{floor diagram}. 
	\end{definition}

	From a parametrised tropical curve, we can obtain a floor diagram by replacing floors with vertices. We obtain a \emph{marked floor diagram} by recording where the conditions $(\ell_i, p_j)$ are met as a function  function $\{1, \ldots, n\} \rightarrow E(\mathcal{D}) \cap V(\mathcal{D})$. A marked floor diagram is characterized by the following properties (we refer to point and line conditions, but the properties can be written purely in terms of the function):
	
	\begin{enumerate}
		\item Each vertex has at least one condition, and a distinguished element in the preimage which is a \emph{horizontal condition}. All other elements in the preimage are \emph{vertical conditions}. 
		\item If a vertex has a point condition, this condition must be a horizontal condition. 
		\item Each edge has at most one point condition, at most two line conditions, and never both a point and line condition. 
		\item Any condition on an edge is a horizontal condition. 
		\item The ordering of the horizontal conditions respect the orientation of the tree $\mathcal{D}$. 
	\end{enumerate}

	\begin{example}
	An example of a marked floor diagram is given in Figure \ref{fig:floordiagram}. Up to equivalence, there are 3 diagrams of this form (given by the changing the ordering of the vertical conditions).

	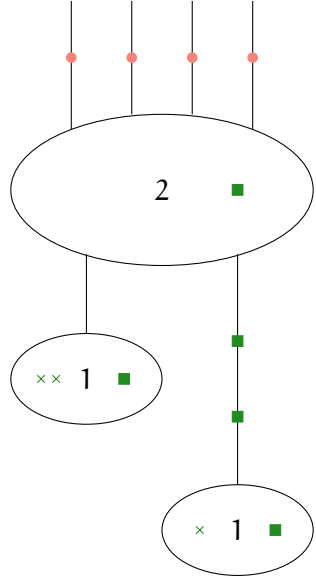
\begin{figure}[h]
		\centering
		\begin{tikzpicture}
			\draw (0,10) ellipse (2 and 1);
			\draw (-1,7.5) ellipse (1 and 0.6);
			\draw (1,5.5) ellipse (1 and 0.6);
			\draw (-1,9.15) -- (-1,8.1);
			\draw (1,9.15) -- (1,6.1);
			
			\draw (1.2,10.8) -- (1.2, 12.5);
			\draw (0.4,11) -- (0.4,12.5);
			\draw (-0.4,11) -- (-0.4,12.5);
			\draw (-1.2,10.8) -- (-1.2,12.5);
			
			\node at (0,10) {$2$};
			\node at (-1,7.5) {$1$};
			\node at (1,5.5) {$1$};
			
			\fill[coralpink] (-1.2, 11.75) circle[radius=2pt];
			\fill[coralpink] (-0.4, 11.75) circle[radius=2pt];
			\fill[coralpink] (0.4, 11.75) circle[radius=2pt];
			\fill[coralpink] (1.2, 11.75) circle[radius=2pt];
			
			\draw [forestgreen] plot [only marks, mark=square*] coordinates {(1,10) (-0.5,7.5) (1.5,5.5) (1, 7) (1,8)};
			\draw [forestgreen] plot [only marks, mark=x] coordinates {(0.5,5.5) (-1.4,7.5) (-1.6,7.5)};
		\end{tikzpicture}
		\caption{Marked floor diagram with $(a,b) = (8,4)$. Vertices are given by ovals with divergence indicated. Point conditions are given by pink circles. Horizontal line conditions are given by green squares, and vertical line conditions by green crosses.}
		\label{fig:floordiagram}
	\end{figure}
	
	\end{example}
	
	Once we have a marked floor diagram, we can try to reconstruct the original tropical curve and count how many ways this can be done (with multiplicity). Given a floor $f$, we can project it to the $x-y$ plane to obtain a tropical curve in $\mathbb{R}^2$ of degree equal to the divergence of this floor. Call this $h_f: \Gamma_f \rightarrow \mathbb{R}^2$. 
	
	The point and line conditions on $\Gamma$ impose point and line conditions on $\Gamma_f$ in the following way. 
	\begin{enumerate}
		\item On each floor, point conditions and vertical conditions impose point conditions. 
		\item A point conditions on an elevator adjacent to the floor $f$ imposes a point condition.
		\item Two line conditions on an elevator adjacent to the floor $f$ imposes a point condition.
		\item One line condition on an elevator adjacent to the floor $f$ will impose a point condition on the floor $f$ if and only if the floors above (if the elevator is above the floor) or below (if the elevator is below the floor) are fixed by conditions \emph{excluding this line}.
		\item All other line conditions do not impose any conditions on this floor. 
	\end{enumerate}
	
	\begin{remark}
		A line condition on an elevator can be viewed as a ``link'' between two floors. It imposes the condition that the two adjacent floors, when projected to $\mathbb{R}^2$ must intersect on a particular line. Meeting a line in $\mathbb{R}^2$ is a trivial condition, but generically a tropical curve will meet a line at a finite number of points. So if one floor is fixed, this creates a point condition on the other floor. 
	\end{remark}

	\begin{definition} 
		Suppose that $(\mathcal{D},m)$ is a marked floor diagram. For each floor $f$, impose point conditions as above. If the number of point conditions on each floor $f$ with divergence $d_f$ is not exactly $3d_f - 1$, then the multiplicity is 0. Otherwise, for each floor, let $\ell_f$ be the sum of: 
		\begin{enumerate}
			\item the number of line conditions on the floor or on adjacent elevators \emph{that do not impose point conditions on the floor}; and
			\item the number of unmarked adjacent elevators.
		\end{enumerate}
		Define 
		$$\mu(f) = d_f^{\ell_f} N_{d_f} $$
		where $N_{d_f}$ is the number of degree-$d_f$ planar curves through $3d_f -1$ points. The \emph{multiplicity} of the floor diagram is 
		$$\mu(\mathcal{D},m) = \prod_{f} \mu(f) \prod_e w(e)^{1+|m^{-1}(e)|}$$
		where the second product is over the elevators of $\mathcal{D}$. 
	\end{definition}

	\begin{theorem}\cite[Theorem 1]{BM07} \label{thm:BM}
		The genus $0$ Gromov--Witten invariant for curves in $\mathbb{P}^3$ with $a$ line and $b$ point conditions is given by
		$$GW_{(a,b)}^0 = \sum_{(\mathcal{D},m)} \mu(\mathcal{D},m)$$
		where the sum is over marked floor diagrams.
	\end{theorem}

	\begin{warning}
		We typically drop the labelling of the infinite ends when counting tropical curves in $\mathbb{P}^3$, which accounts for a factor of $(4d)!$. 
	\end{warning}
	
	This can be modified for our needs by inputting the modified invariants for $h_f \rightarrow \mathbb{R}^2$ (where each vertex is weighed by $\frac{1}{24} |u_1 \wedge u_2|^2$) into the recursion.

	\subsection{Floor diagrams in genus $1$}
	
	We have a natural definition of a floor diagram in higher genus. 
	
	\begin{definition}
		Let $\mathcal{D}$ be a oriented graph with $d$ infinite ends which are oriented outwards. Impose that the divergence of each vertex is positive. Let $g: V(\mathcal{D}) \rightarrow \mathbb{Z}_{\geq 0}$ be a genus function. Let $g = \sum_{v \in V(\mathcal{D})} g(v)$ and call $(\mathcal{D},g)$ a \emph{generalised floor diagram of genus $g$}. 
		
		A \emph{marked (generalised) floor diagram} is a floor diagram $(\mathcal{D},g)$ with marking function 
		$$m: \{1, \ldots, n\} \rightarrow V(\mathcal{D}) \cup E(\mathcal{D}).$$
	\end{definition}

	One difficulty in the positive genus case is that it is no longer true that tropical curves meeting sufficiently vertically stretched conditions are floor decomposed. A key lemma that fails is the following.\footnote{This lemma appears in an unpublished version of \cite{BM07}.} 
	
	\begin{lemma}[Brugall{\'e}--Mikhalkin]
		Suppose that there are $a$ tropical line and $b$ point constraints, with $4d=a+2b$, with vertices contained in the region $[-A,A]^2 \times \mathbb{R}$. Then for any genus 0 tropical curve $(h,\Gamma)$ passing through these conditions, for all vertices $v$ of $\Gamma$, we have $h(v) \in [-A,A]^2 \times \mathbb{R}$. 
	\end{lemma}
	
	The proof of the lemma requires that there are no unexpected dimensions of deformation, and this fails because of superabundancy. Thus we must allow horizontal line conditions to occupy the same floor as other horizontal conditions. It is significantly harder to compute multiplicities (even without the issue of superabundancy, since floor diagrams do not encode enough information to know if a loop can close up), but we can still use marked floor diagrams to list the relevant tropical types. We give examples of this procedure in proceeding sections. 
	
	\begin{remark}
		We speculate that the cases that appear can be severely constrained. We do not explore this here. 
	\end{remark}

	\subsection{Counts in low degree}
	
	We prove Corollary \ref{cor:logandordinary}. The difference in the Gromov--Witten count $GW_{a,b}$ and logarithmic Gromov--Witten count $L_{a,b}$ is controlled by the terms $E_0$ and $E_1$. These are the modified genus 0 count and the excess well-spaced count respectively. More precisely, (up to a factor of $(d!)^4$ given by labelling the markings corresponding to infinite ends), the difference is $E_1 + (E_0 - GW_{a,b}^0)$.
	
	\begin{proposition} \label{prop:phase}
		Let $E_0,E_1$ be as above. Then the difference in Gromov--Witten and logarithmic Gromov--Witten count (when corrected by a factor of $(d!)^4$) is positive precisely for the $(a,b)$ given in Figure  \ref{fig:phasetransition}. Here $d \in \mathbb{Z}_{>0}, a,b \in \mathbb{Z}_{\geq 0}, a+2b = 4d$. 
	\end{proposition}
	
	\begin{figure}[h]
		\centering
		\begin{tikzpicture}
			
			\fill[coralpink] (2,-3) rectangle ++(1,1);
			\fill[coralpink] (0,-4) rectangle ++(1,1);
			
			  \foreach \i in {1,...,5}{
			  	\node at (-0.5,-\i+0.5) {$\i$};
			  	\pgfmathsetmacro{\endloop}{2*\i};
			  	\foreach \j in {0,...,\endloop}{
			  		\pgfmathsetmacro{\transition}{int(\j +2*\i)};
			  		\ifthenelse{\transition < 9}{
		  				}{
			  			\fill[coralpink] (\j, -\i) rectangle ++(1,1);
		  			}
					\draw[draw=black] (\j, -\i) rectangle ++(1,1);
				}
			}
			\foreach \j in {0,...,10}{
			\pgfmathsetmacro{\a}{int(2*\j)};
			\node at (\j+0.5, 0.5) {$\a$};
			}

			\draw[draw=black] (3,-3) rectangle ++(1,1);

			\draw[->] (0,1) -- (1,1);
			\node at (0.5,1.3) {$a$};
			\draw[->] (-1,0) -- (-1,-1);
			\node at (-1.3,-0.5) {$d$};
		\end{tikzpicture}
		\caption{Superabundant and well-spaced contributions are shown shaded.}
		\label{fig:phasetransition}
	\end{figure}
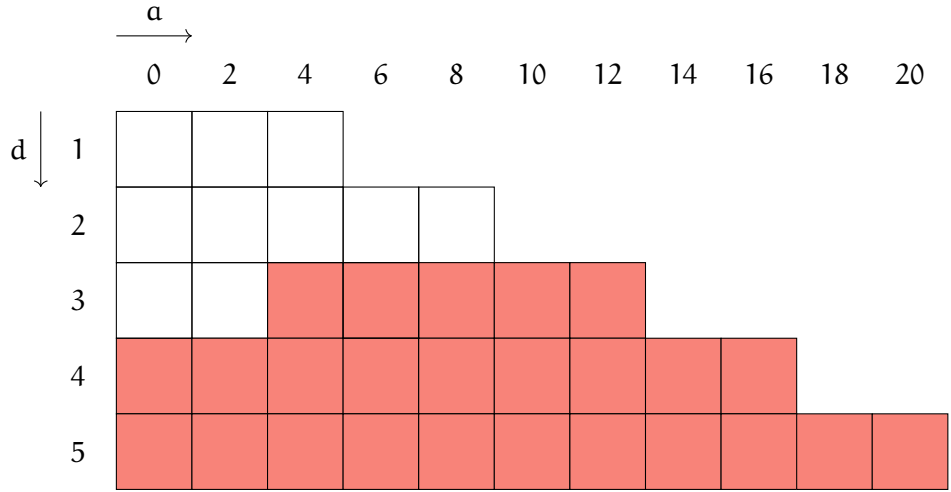
	
	The rest of this section is a proof of Proposition \ref{prop:phase}. For convenience, define
	$$\tilde{E}_0 = E_0 - (d!)^4GW_{a,b}^0.$$
	We can split the contribution $E_1$ into $E_{1,1} + E_{1,2}$ coming from superabundant curves where the span of the cycle is $1$ and $2$ respectively. 
	
	\begin{lemma}
		The difference $\tilde{E}_0$ is positive if and only if there is a marked genus 0 floor diagram of positive multiplicity with either
		\begin{enumerate}
			\item a floor of divergence 3; or 
			\item a weight 2 elevator. 
		\end{enumerate}
		Moreover $E_{1,1}$ is positive only if the above condition holds.  
	\end{lemma}

	\begin{proof}
		Fix $d, a, b$ such that $a + 2b = 4d$. Suppose that no marked genus 0 floor diagram of degree $d$ with $a$ line and $b$ point conditions has a floor of divergence $3$ nor a weight $2$ elevator. There is only one combinatorial type of planar tropical curve of degree 1 (and respectively degree 2) passing through 2 (respectively 5) point conditions. Both of these tropical types has the property that all the vertex multiplicities every edge has weight $1$, and every edge has direction vector (up to sign) one of 
		$$\{e_1,e_2,e_3, e_1\pm e_2, e_1 \pm e_3, e_2 \pm e_3\}.$$
		
		The vertices of the tropical curve can be read off from its floor diagram. They are
		\begin{enumerate}
			\item the vertices coming from the (non-divalent) vertices of $(h_f, \Gamma_f)$, for each floor $f$; and
			\item the vertices on the ends of an elevator $e$, for each elevator $e$.
		\end{enumerate}
		These two sets of vertices are disjoint by generality. For the vertices of type (1), the multiplicity $|a_1 b_2 - a_2 b_1|$ does not depend on the last coordinate. So if there are no floors of divergence greater than 2, all of the vertex multiplicities are 1. For the vertices of type (2), suppose that the elevator $e$ is connected to a edges $f_1, f_2$ in the adjacent floor, coming from an edge $f$ in $(h_f,\Gamma_f)$. Say $f$ has direction vector $u = (u_1, u_2)$. We have assumed that the weight of the edge $e$ is 1, so the directions of the edges adjacent to this vertex are $\pm e_3, (-u_1, -u_2, k)$, and $(u_1, u_2, k \mp 1)$. Considering all possible $u$ as listed above, we observe that the vertex multiplicity is always $1$. 
		
		Conversely, it is clear that if there is a weight 2 elevator, then the vertex multiplicity of the adjacent vertices is greater than $1$. If there is a floor of divergence $3$, the combinatorial type given in Figure \ref{fig:degree3planar} will appear in a floor, again producing a vertex with vertex multiplicity greater than 1. The result for $E_0$ follows. 
		
			\begin{figure}[h]
			\centering
			\begin{tikzpicture}[scale=0.5]
				\draw (0,0) -- (-2,0);
				\draw (-2,0) -- (-3,1);
				\draw (-2,0) -- (-3,1) -- (-4,1) -- (-5,0);
				\draw (-3,1) -- (-3,2); 
				\draw (-4,1) -- (-4,2);
				\draw (-2,0) -- (-3,-1);
				
				\draw (0,1) -- (0,0) -- (2,-1) -- (3,-1);
				\draw (2,-1) -- (3,-2) -- (4,-2);
				\draw (3,-2) -- (3,-3) -- (4,-3); 
				\draw (3,-3) -- (2,-4);

				\node at (-1,0.3) {$2$};
			\end{tikzpicture}
			\caption{A degree $3$ genus 0 tropical curve with non-trivial vertex multiplicity.}
			\label{fig:degree3planar}
		\end{figure}
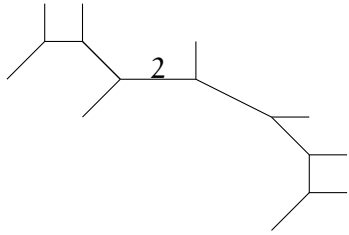
		
		For $E_{1,1}$, consider a tropical curve meeting the conditions where the span of the cycle is $1$ dimensional. As in the proof of Lemma \ref{lem:simplify}, we work with the genus $0$ curve with the same image. We obtain a genus 0 floor diagram. However, the cycle of the original curve maps to an edge in the genus 0 graph of weight at least 2. The result follows as before. 
	\end{proof}

	Call this condition $(\star)$. We find when condition $(\star)$ holds. For degree $<2$, clearly the condition doesn't hold. For degree $3$ and $4$, it is straightforward check by hand. For degree $d = 5$ and $a = 0$ (all point conditions), the floor diagram given in Figure \ref{fig:degree5} shows that the condition holds. 
	
		\begin{figure}[h]
		\centering
		\begin{tikzpicture}[scale=0.5]
			\draw (0,10) ellipse (2 and 1);
			\draw (-1,7.5) ellipse (1 and 0.6);
			\draw (1,5.5) ellipse (1 and 0.6);
			\draw (-1,9.15) -- (-1,8.1);
			\draw (1,9.15) -- (1,6.1);
			
			\draw (0,11) -- (0,12.5);
			\draw (0.6,10.95) -- (0.6,12.5);
			\draw (1.2,10.8) -- (1.2,12.5);
			\draw (-0.6,10.95) -- (-0.6,12.5);
			\draw (-1.2,10.8) -- (-1.2,12.5);
			
			\node at (0,10) {$3$};
			\node at (-1,7.5) {$1$};
			\node at (1,5.5) {$1$};
			
			\fill[coralpink] (-1.2, 11.75) circle[radius=4pt];
			\fill[coralpink] (-0.6, 11.75) circle[radius=4pt];
			\fill[coralpink] (0, 11.75) circle[radius=4pt];
			\fill[coralpink] (0.6, 11.75) circle[radius=4pt];
			\fill[coralpink] (1.2, 11.75) circle[radius=4pt];
			\fill[coralpink] (-1, 8.6) circle[radius=4pt];
			\fill[coralpink] (1, 7.5) circle[radius=4pt];
			\fill[coralpink] (-1.5,7.5) circle[radius=4pt];
			\fill[coralpink] (0.5,5.5) circle[radius=4pt];
			\fill[coralpink] (1, 10) circle[radius=4pt];
		\end{tikzpicture}
		\caption{Marked floor diagram with $(a,b) = (10,0)$ with point conditions are given by pink circles.}
		\label{fig:degree5}
	\end{figure}
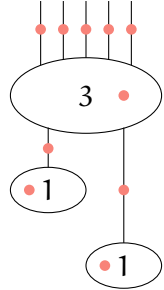

	We can use this to show that the condition holds for $d \geq 5, a = 0$ by using the procedure shown in Figure \ref{fig:pointconditions}. Each step increases the degree by $1$. 
	
		\begin{figure}[h]
		\centering
		\begin{subfigure}{0.18\textwidth}
			\begin{tikzpicture}
				\draw (0,10) ellipse (1 and 0.6);
				\draw (0.8,8) ellipse (0.6 and 0.4);
				\draw (-0.8,6) ellipse (0.6 and 0.4);
				
				\draw (0,10.6) -- (0,11.5);
				\draw (0.3,10.55) -- (0.3,11.5);
				\draw (-0.3,10.55) -- (-0.3,11.5);
				\draw (0.6,10.48) -- (0.6,11.5);
				\draw (-0.6,10.48) -- (-0.6,11.5);
				
				\draw (0.6,9.5) -- (0.6,8.37); 
				
				\draw (-0.6,9.5) -- (-0.6,6.37);
				
				\node at (0,10) {$3$};
				\node at (0.8,8) {$1$};
				\node at (-0.8,6) {$1$};
				
				
				\foreach \coord in {(0,11), (0.3, 11), (-0.3, 11), (0.6, 11), (-0.6, 11), (0.6,10), (1.1, 8), (-0.5, 6), (0.6,9), (-0.6, 9)}{
					\fill[coralpink] \coord circle[radius=2pt];
				}
				
				\node at (0,4.5) {$d=5$};
			\end{tikzpicture}
		\end{subfigure}
		\begin{subfigure}{0.18\textwidth}
			\begin{tikzpicture}
				\draw (0,10) ellipse (1 and 0.6);
				\draw (0.8,8) ellipse (0.6 and 0.4);
				\draw (-0.8,6) ellipse (0.6 and 0.4);
				\draw (-0.8,7.5) ellipse (0.6 and 0.4);
				
				\draw (0,10.6) -- (0,11.5);
				\draw (0.3,10.55) -- (0.3,11.5);
				\draw (-0.3,10.55) -- (-0.3,11.5);
				\draw (0.6,10.48) -- (0.6,11.5);
				\draw (-0.6,10.48) -- (-0.6,11.5);
				\draw (0.9, 10.25) -- (0.9,11.5);
				
				\draw (0.6,9.5) -- (0.6,8.37); 
				
				\draw (-0.6,9.5) -- (-0.6,7.87);
				\draw (-0.6, 6.37) -- (-0.6, 7.13);
				
				\node at (0,10) {$3$};
				\node at (0.8,8) {$1$};
				\node at (-0.8,6) {$1$};
				\node at (-0.8,7.5) {$1$};
				
				
				\foreach \coord in {(0,11), (0.3, 11), (-0.3, 11), (0.6, 11), (-0.6, 11), (0.6,10), (1.1, 8), (-0.5, 6), (0.6,9), (-0.5, 7.5), (0.9, 11), (-0.6, 6.75)}{
					\fill[coralpink] \coord circle[radius=2pt];
				}
				\node at (0,4.5) {$d=6$};
				\node at (-0.7,8.5) {$2$};
			\end{tikzpicture}
		\end{subfigure}
		\begin{subfigure}{0.2\textwidth}
			\begin{tikzpicture}
				\draw (0,10) ellipse (1 and 0.6);
				\draw (0.8,8) ellipse (0.6 and 0.4);
				\draw (-0.8,7.5) ellipse (0.6 and 0.4);
				\draw (-0.8,6) ellipse (0.6 and 0.4);
				
				\draw (0,10.6) -- (0,11.5);
				\draw (0.3,10.55) -- (0.3,11.5);
				\draw (-0.3,10.55) -- (-0.3,11.5);
				\draw (0.6,10.48) -- (0.6,11.5);
				\draw (-0.6,10.48) -- (-0.6,11.5);
				
				\draw (-1.1,7.85) -- (-1.1,11.5);
				\draw (-1.3,7.7)  -- (-1.3, 11.5);
				
				\draw (0.6,9.5) -- (0.6,8.37); 
				
				\draw (-0.6,9.5) -- (-0.6,7.87);
				\draw (-0.6, 6.37) -- (-0.6, 7.13);
				
				\node at (0,10) {$3$};
				\node at (0.8,8) {$1$};
				\node at (-0.8,7.5) {$2$};
				\node at (-0.8,6) {$1$};
				
				
				\foreach \coord in {(0,11), (0.3, 11), (-0.3, 11), (0.6, 11), (-0.6, 11), (0.6,10), (1.1, 8), (-0.5, 6), (0.6,9), (-0.6, 9), (-1.1, 11), (-1.3, 11), (-0.5, 7.5), (-0.6, 6.75)}{
					\fill[coralpink] \coord circle[radius=2pt];
				}
				
				\node at (0,4.5) {$d=7$};
			\end{tikzpicture}
		\end{subfigure}
		\begin{subfigure}{0.2\textwidth}
			\begin{tikzpicture}
				\draw (0,10) ellipse (1 and 0.6);
				\draw (0.8,8) ellipse (0.6 and 0.4);
				\draw (-0.8,7.5) ellipse (0.8 and 0.6);
				\draw (-0.8,6) ellipse (0.6 and 0.4);
				\draw (-0.8, 4.5) ellipse (0.6 and 0.4);
				
				\draw (0,10.6) -- (0,11.5);
				\draw (0.3,10.55) -- (0.3,11.5);
				\draw (-0.3,10.55) -- (-0.3,11.5);
				\draw (0.6,10.48) -- (0.6,11.5);
				\draw (-0.6,10.48) -- (-0.6,11.5);
				
				\draw (-1.1,8.05) -- (-1.1,11.5);
				\draw (-1.3,7.95)  -- (-1.3, 11.5);
				\draw (-1.5, 7.8) -- (-1.5, 11.5);
				
				\draw (0.6,9.5) -- (0.6,8.37); 
				
				\draw (-0.6,9.5) -- (-0.6,8.08);
				\draw (-0.6,6.37) -- (-0.6, 6.91);
				\draw (-0.6,5.63) -- (-0.6,4.87);
				
				\node at (0,10) {$3$};
				\node at (0.8,8) {$1$};
				\node at (-0.8,7.5) {$2$};
				\node at (-0.8,6) {$1$};
				\node at (-0.8, 4.5) {$1$};
				
				
				\foreach \coord in {(0,11), (0.3, 11), (-0.3, 11), (0.6, 11), (-0.6, 11), (0.6,10), (1.1, 8), (-0.5, 6), (0.6,9), (-0.6, 9), (-1.1, 11), (-1.3, 11), (-0.5, 7.5), (-0.6, 5.3), (-1.5, 11), (-0.5, 4.5)}{
					\fill[coralpink] \coord circle[radius=2pt];
				}
				\node at (-0.9, 6.65) {$2$};
				\node at (0.7,4.5) {$d=8$};
			\end{tikzpicture}
		\end{subfigure}
		\begin{subfigure}{0.18 \textwidth}
			\begin{tikzpicture}
				\draw (0,10) ellipse (1 and 0.6);
				\draw (0.8,8) ellipse (0.6 and 0.4);
				\draw (-0.8,7.5) ellipse (0.8 and 0.6);
				\draw (-0.6,6) ellipse (0.6 and 0.4);
				\draw (-1, 4.5) ellipse (0.6 and 0.4);
				
				\draw (0,10.6) -- (0,11.5);
				\draw (0.3,10.55) -- (0.3,11.5);
				\draw (-0.3,10.55) -- (-0.3,11.5);
				\draw (0.6,10.48) -- (0.6,11.5);
				\draw (-0.6,10.48) -- (-0.6,11.5);
				
				\draw (-1.1,8.05) -- (-1.1,11.5);
				\draw (-1.24,7.98)  -- (-1.24, 11.5);
				\draw (-1.36,7.94)  -- (-1.36, 11.5);
				\draw (-1.5, 7.8) -- (-1.5, 11.5);
				
				\draw (0.6,9.5) -- (0.6,8.37); 
				
				\draw (-0.6,9.5) -- (-0.6,8.08);
				\draw (-0.6,6.37) -- (-0.6, 6.91);
				
				\draw (-1.3,7.05) -- (-1.3,4.85);
				
				\node at (0,10) {$3$};
				\node at (0.8,8) {$1$};
				\node at (-0.8,7.5) {$3$};
				\node at (-0.6,6) {$1$};
				\node at (-1, 4.5) {$1$};
				
				
				\foreach \coord in {(0,11), (0.3, 11), (-0.3, 11), (0.6, 11), (-0.6, 11), (0.6,10), (1.1, 8), (-0.3, 6), (0.6,9), (-0.6, 9), (-1.1, 11), (-1.24, 11), (-1.36, 11), (-0.5, 7.5), (-1.5, 11), (-0.7, 4.5), (-1.3, 6.6), (-0.6, 6.6)}{
					\fill[coralpink] \coord circle[radius=2pt];
				}

				\node at (0.7,4.5) {$d=9$};
			\end{tikzpicture}
		\end{subfigure}
		\caption{Marked floor diagrams for all point conditions. The procedure can be repeated for each degree $3$ floor that appears. By induction, we obtain marked floor diagrams for $d \geq 5$.}
		\label{fig:pointconditions}
	\end{figure}
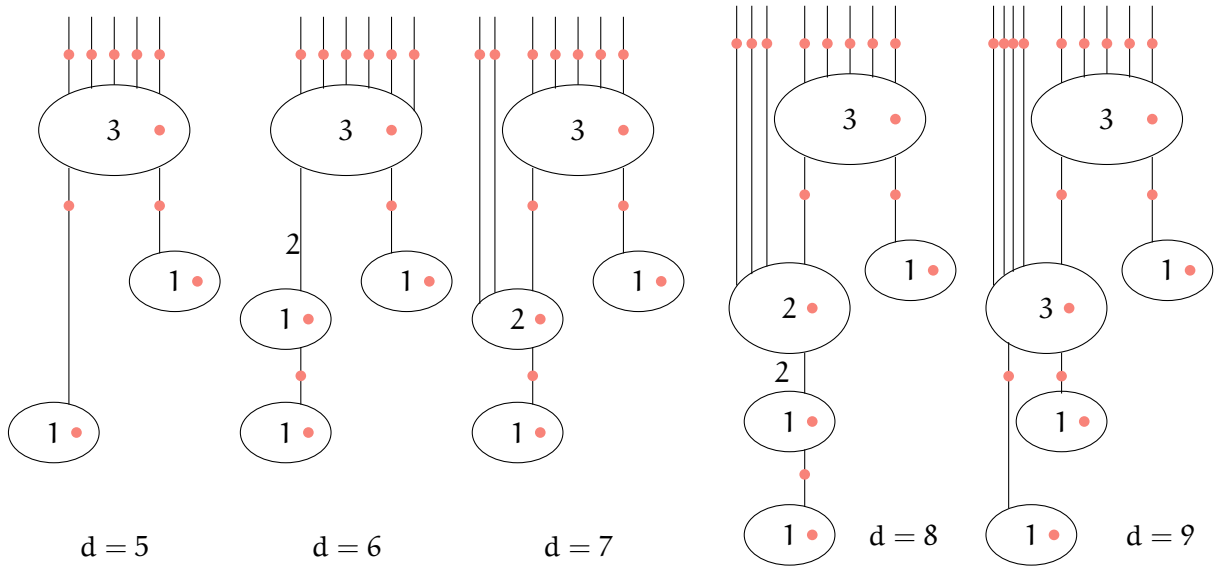

	\begin{lemma}
		If condition $(\star)$ holds for $(a,b)$, with $b>0$, then it holds for $(a+2,b-1)$. 
	\end{lemma}
	
	\begin{proof}
		Given a marked floor diagram with $a$ line and $b$ point conditions, with $b$ positive, take the point condition which is highest in the ordering and replace it with two lines. This gives the required marked floor diagram for $a+2$ lines and $b-1$ point conditions. 
	\end{proof}
	
	\begin{corollary}
		Condition $(\star)$ holds for $d \geq 5$. 
	\end{corollary}
	
	Finally, we want to show that when condition $(\star)$ doesn't hold, the contribution $E_{1,2}$ is $0$. Note that a genus $1$ planar curve with genus not at vertices must have degree at least $3$. So if there are no divergence $3$ vertices in the marked floor diagram, there cannot be an open cycle in a floor. It remains to check for cycles that go across floors. This means the underlying graph of $\mathcal{D}$ must have genus 1, which immediately eliminates degree $1$ and $2$.

	We must check what happens in degree $3$. For degree $3$, the relevant floor diagrams are shown in Figure \ref{fig:acrossfloors}. They do not create extra contributions for $b \geq 3$. Indeed, point conditions must be horizontal conditions above any horizontal line conditions. Since the excess dimension is $1$, there can be at most one floor or infinite end without a horizontal condition (each such implies there is an extra dimension of deformation). This implies there are at least three point conditions adjacent to the top floor, which has divergence $1$, giving a contradiction. 
	
	\begin{figure}[h]
		\centering
		\begin{tikzpicture}[scale=0.6]
			\draw (0,10) ellipse (1 and 0.6);
			\draw (0,8) ellipse (1 and 0.6);
			\draw (0,6) ellipse (1 and 0.6);

			\draw (0,10.6) -- (0,11.5);
			\draw (0.6,10.5) -- (0.6,11.5);
			\draw (-0.6,10.5) -- (-0.6,11.5);

			\draw (0.2,9.4) -- (0.2,8.6); 
			\draw (-0.2,9.4) -- (-0.2,8.6);
			
			\draw (0,6.6) -- (0,7.4);
			
			\node at (0,10) {$1$};
			\node at (0,8) {$1$};
			\node at (0,6) {$1$};

			\draw (3,10) ellipse (1 and 0.6);
			\draw (3,8) ellipse (1 and 0.6);
			
			\draw (3,10.6) -- (3,11.5);
			\draw (3.6,10.5) -- (3.6,11.5);
			\draw (2.4,10.5) -- (2.4,11.5);
			
			\draw (3.2,9.4) -- (3.2,8.6); 
			\draw (2.8,9.4) -- (2.8,8.6);
			
			\node at (3,10) {$1$};
			\node at (3,8) {$2$};

		\end{tikzpicture}
		\caption{Degree $3$ floor diagrams with a cycle across floors.}
		\label{fig:acrossfloors}
	\end{figure}
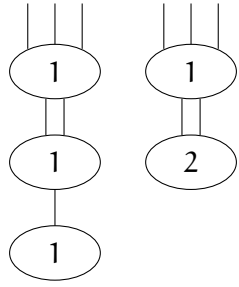

	\subsection{Full example} We end with the full calculation for degree $d=4$ and $(a,b) = (2,7)$. 
	
	In genus 0, the relevant marked floor diagrams are given in Figure \ref{fig:(2,7)genus0contributions}, with multiplicity and number or possible orderings of the conditions shown. Without modification, this gives a total of 58, which is the genus $0$ Gromov--Witten invariant. For the $E_0$ contribution, there are two vertices of multiplicity 4 in two of the floor diagrams. This gives
	$$\frac{1}{(4!)^4}E_0 = 860.$$

	\begin{figure}[h]
		\centering
		\begin{subfigure}{0.18\textwidth}
			\begin{tikzpicture}
				\draw (0,10) ellipse (1 and 0.6);
				\draw (0,8) ellipse (1 and 0.6);
				\draw (0,6) ellipse (1 and 0.6);
				
				\draw (0.2,10.6) -- (0.2,11.5);
				\draw (-0.2,10.6) -- (-0.2,11.5);
				\draw (0.6,10.5) -- (0.6,11.5);
				\draw (-0.6,10.5) -- (-0.6,11.5);
				
				\draw (0,9.4) -- (0,8.6); 
				
				\draw (0,6.6) -- (0,7.4);
				
				\node at (0,10) {$2$};
				\node at (0,8) {$1$};
				\node at (0,6) {$1$};
				
				\node at (-0.3, 9) {$2$};
				
				
				\foreach \coord in {(0.2, 11), (-0.2, 11), (0.6, 11), (-0.6, 11), (0.6, 8), (0, 7), (0.6,10)}{
					\fill[coralpink] \coord circle[radius=2pt];
				}
				\draw [forestgreen] plot [only marks, mark=square*] coordinates {(0.7, 6)};
				\draw [forestgreen] plot [only marks, mark=x] coordinates {(0.4, 6)};

				\node at (0,4.5) {$\mu = 4, n = 1$};
			\end{tikzpicture}
		\end{subfigure}
		\begin{subfigure}{0.18\textwidth}
			\begin{tikzpicture}
				\draw (0,10) ellipse (1 and 0.6);
				\draw (0,8) ellipse (1 and 0.6);
				\draw (0,6) ellipse (1 and 0.6);
				
				\draw (0.2,10.6) -- (0.2,11.5);
				\draw (-0.2,10.6) -- (-0.2,11.5);
				\draw (0.6,10.5) -- (0.6,11.5);
				\draw (-0.6,10.5) -- (-0.6,11.5);
				
				\draw (0,9.4) -- (0,8.6); 
				
				\draw (0,6.6) -- (0,7.4);
				
				\node at (0,10) {$2$};
				\node at (0,8) {$1$};
				\node at (0,6) {$1$};
				
				\node at (-0.3, 9) {$2$};
				
				
				\foreach \coord in {(0.2, 11), (-0.2, 11), (0.6, 11), (-0.6, 11), (0.8, 8), (0.8, 6), (0.6,10)}{
					\fill[coralpink] \coord circle[radius=2pt];
				}
				\draw [forestgreen] plot [only marks, mark=x] coordinates {(0.4, 8) (0.4, 6)};

				\node at (0,4.5) {$\mu = 2, n =  2$};
			\end{tikzpicture}
		\end{subfigure}
		\begin{subfigure}{0.18\textwidth}
			\begin{tikzpicture}
				\draw (0,10) ellipse (1 and 0.6);
				\draw (0.8,8) ellipse (0.6 and 0.4);
				\draw (-0.8,6) ellipse (0.6 and 0.4);
				
				\draw (0.2,10.6) -- (0.2,11.5);
				\draw (-0.2,10.6) -- (-0.2,11.5);
				\draw (0.6,10.5) -- (0.6,11.5);
				\draw (-0.6,10.5) -- (-0.6,11.5);
				
				\draw (0.6,9.5) -- (0.6,8.37); 
				
				\draw (-0.6,9.5) -- (-0.6,6.37);
				
				\node at (0,10) {$2$};
				\node at (0.8,8) {$1$};
				\node at (-0.8,6) {$1$};
				
				
				\foreach \coord in {(0.2, 11), (-0.2, 11), (0.6, 11), (-0.6, 11), (0.6,10), (1.1, 8), (-0.5, 6)}{
					\fill[coralpink] \coord circle[radius=2pt];
				}
				\draw [forestgreen] plot [only marks, mark=x] coordinates {(0.5, 8) (-1.1, 6)};

				\node at (0,4.5) {$\mu = 4, n = 2$};
			\end{tikzpicture}
		\end{subfigure}
		\begin{subfigure}{0.18\textwidth}
		\begin{tikzpicture}
			\draw (0,10) ellipse (0.8 and 0.5);
			\draw (0.8,8) ellipse (0.6 and 0.4);
			
			\draw (0,10.5) -- (0,11.5);
			\draw (0.6,10.33) -- (0.6,11.5);
			\draw (-0.6,10.33) -- (-0.6,11.5);
			
			\draw (0.6,9.65) -- (0.6,8.37); 
			
			\draw (1, 8.37)  -- (1,11.5);

			\node at (0,10) {$2$};
			\node at (0.8,8) {$2$};
			
			
			\foreach \coord in {(0, 11), (0.6, 11), (-0.6, 11), (0.5,10), (1.1, 8), (1,11), (0.6, 9)}{
				\fill[coralpink] \coord circle[radius=2pt];
			}
			\draw [forestgreen] plot [only marks, mark=x] coordinates {(0.4, 8) (0.6,8)};

			\node at (0,4.5) {$\mu = 4, n = 8$};
		\end{tikzpicture}
		\end{subfigure}
		\begin{subfigure}{0.18 \textwidth}
			\begin{tikzpicture}
				\draw (0,10) ellipse (0.8 and 0.5);
				\draw (0.8,8) ellipse (0.6 and 0.4);
				\draw (0.8,6) ellipse (0.6 and 0.4);
				
				\draw (0,10.5) -- (0,11.5);
				\draw (0.6,10.33) -- (0.6,11.5);
				\draw (-0.6,10.33) -- (-0.6,11.5);
				
				\draw (0.6,9.65) -- (0.6,8.37); 
				
				\draw (1, 8.37)  -- (1,11.5);
				
				\draw (0.8, 7.6)  -- (0.8, 6.4);

				\node at (0,10) {$2$};
				\node at (0.8,8) {$1$};
				\node at (0.8,6) {$1$};
				
				
				\foreach \coord in {(0, 11), (0.6, 11), (-0.6, 11), (0.5,10), (1.1, 8), (1,11), (1.1, 6)}{
					\fill[coralpink] \coord circle[radius=2pt];
				}
				\draw [forestgreen] plot [only marks, mark=x] coordinates {(-0.5, 10) (0.5,6)};

				\node at (0,4.5) {$\mu = 2, n = 5$};
			\end{tikzpicture}
		\end{subfigure}
		\caption{Marked floor diagrams of genus $0$ for $(a,b) = (2,7)$ with non-zero multiplicity. The number $n$ is the number of possible orderings of the markings. The multiplicity given is the genus $0$ multiplicity. In the $E_0$ count, the two leftmost multiplicities are modified.}
		\label{fig:(2,7)genus0contributions}
	\end{figure}
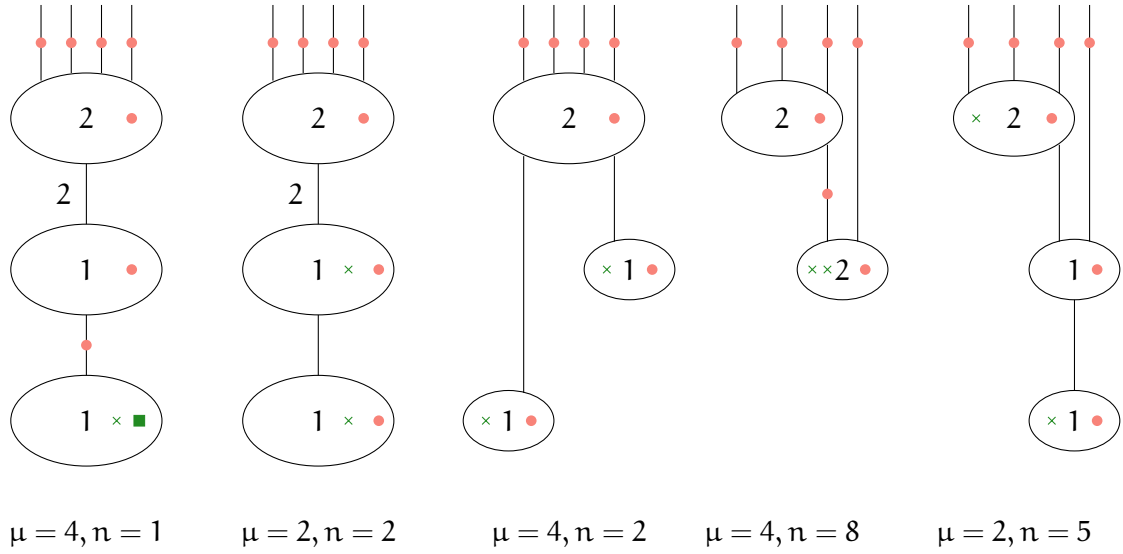

	There are no $E_{1,1}$ contributions. Indeed, consider the leftmost graph of Figure \ref{fig:(2,7)genus0contributions}. We can put a cycle on the weight 2 edge, but by Remark \ref{rem:empty}, the corresponding moduli is empty and there is no contribution to the well-spaced count.

	For tropical curves with an open cycle, the relevant floor diagrams are given in Figure \ref{fig:(2,7)genus1contributions}. All contributions are non-superabundant. The cycle multiplicity is $1$, and this gives
	$$\frac{1}{(4!)^4}W_{2,7} = 2 + 2 = 4,$$
	as expected. 
	
	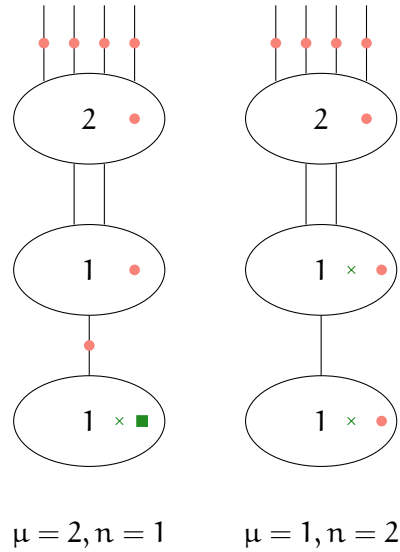
\begin{figure}[h]
		\centering
		\begin{subfigure}{0.18\textwidth}
			\begin{tikzpicture}
				\draw (0,10) ellipse (1 and 0.6);
				\draw (0,8) ellipse (1 and 0.6);
				\draw (0,6) ellipse (1 and 0.6);
				
				\draw (0.2,10.6) -- (0.2,11.5);
				\draw (-0.2,10.6) -- (-0.2,11.5);
				\draw (0.6,10.5) -- (0.6,11.5);
				\draw (-0.6,10.5) -- (-0.6,11.5);
				
				\draw (-0.2,9.4) -- (-0.2,8.6); 
				\draw (0.2, 9.4) -- (0.2, 8.6);
				
				\draw (0,6.6) -- (0,7.4);
				
				\node at (0,10) {$2$};
				\node at (0,8) {$1$};
				\node at (0,6) {$1$};

				
				\foreach \coord in {(0.2, 11), (-0.2, 11), (0.6, 11), (-0.6, 11), (0.6, 8), (0, 7), (0.6,10)}{
					\fill[coralpink] \coord circle[radius=2pt];
				}
				\draw [forestgreen] plot [only marks, mark=square*] coordinates {(0.7, 6)};
				\draw [forestgreen] plot [only marks, mark=x] coordinates {(0.4, 6)};

				\node at (0,4.5) {$\mu = 2, n = 1$};
			\end{tikzpicture}
		\end{subfigure}
		\begin{subfigure}{0.18\textwidth}
			\begin{tikzpicture}
				\draw (0,10) ellipse (1 and 0.6);
				\draw (0,8) ellipse (1 and 0.6);
				\draw (0,6) ellipse (1 and 0.6);
				
				\draw (0.2,10.6) -- (0.2,11.5);
				\draw (-0.2,10.6) -- (-0.2,11.5);
				\draw (0.6,10.5) -- (0.6,11.5);
				\draw (-0.6,10.5) -- (-0.6,11.5);
				
				\draw (-0.2,9.4) -- (-0.2,8.6); 
				\draw (0.2, 9.4) -- (0.2, 8.6);
				
				\draw (0,6.6) -- (0,7.4);
				
				\node at (0,10) {$2$};
				\node at (0,8) {$1$};
				\node at (0,6) {$1$};

				
				\foreach \coord in {(0.2, 11), (-0.2, 11), (0.6, 11), (-0.6, 11), (0.8, 8), (0.8, 6), (0.6,10)}{
					\fill[coralpink] \coord circle[radius=2pt];
				}
				\draw [forestgreen] plot [only marks, mark=x] coordinates {(0.4, 8) (0.4, 6)};

				\node at (0,4.5) {$\mu = 1, n =  2$};
			\end{tikzpicture}
		\end{subfigure}

		\caption{Marked floor diagrams of genus $1$ for $(a,b) = (2,7)$ with non-zero multiplicity. .}
		\label{fig:(2,7)genus1contributions}
	\end{figure}
	
	We have 
	$$L_{2,7} = W_{2,7} - \frac{1}{24}E_0 = (4!)^4\frac{-191}{6}.$$

	\bibliographystyle{siam} 
	\bibliography{bibliographySK} 

\end{document}